\documentclass{amsart}

\usepackage{amssymb}
\usepackage{bm}
\usepackage{graphicx}
\usepackage[centertags]{amsmath}
\usepackage{amsfonts}
\usepackage{amsthm}

\newtheorem{thm}{Theorem}
\newtheorem{cor}{Corollary}
\newtheorem{lem}{Lemma}

\newtheorem{prop}{Proposition}
\newtheorem{defn}{Definition}
\theoremstyle{definition}
\newtheorem{rem}{Remark}

\newtheorem{notation}{Notation}

\newcommand{\rr}{\mathbb{R}}
\newcommand{\nn}{\mathbb{N}}
\newcommand{\qq}{\mathbb{Q}}

\newcommand{\ee}{\varepsilon}

\newcommand{\sxi}{\mathcal{S}_{\xi}}
\newcommand{\sg}{\sigma}

\newcommand{\lesseq}{\sqsubseteq}

\newcommand{\I}{\mathrm{I}}

\newcommand{\spp}{\mathrm{supp}}
\newcommand{\rg}{\mathrm{ran}}

\newcommand{\xfr}{\mathfrak{X}}

\newcommand{\ttt}{\mathcal{T}}
\newcommand{\aaa}{\mathcal{A}}

\newcommand{\lblo}{\prec^l}
\newcommand{\auxa}{\xfr_0}
\newcommand{\auxb}{\xfr_1}

\newcommand{\auxx}{\xfr_{\xi}}
\newcommand{\go}{\mathrm{G}_0}
\newcommand{\goo}{\mathrm{G}_1}
\newcommand{\gox}{\mathrm{G}_\xi}
\newcommand{\kkk}{\mathrm{K}}
\newcommand{\fr}{\frac{1}{2^n}}
\newcommand{\ttk}{\ttt_{\kkk}}
\newcommand{\couo}{\xi<\omega_1}
\newcommand{\con}{$c_{00}(\nn)\text{ }$}
\newcommand{\gft}{$G_{\theta, \mathcal{F}}\text{ }$}

\newcommand{\noe}{\neq\emptyset}
\newcommand{\gxn}{G_n}

\newcommand{\lesss}{\prec}

\newcommand{\sgn}{\mathrm{sgn}}
\newcommand{\segm}{\mathrm{seg}}
\newcommand{\tr}{\mathrm{T}}
\newcommand{\scu}{\mathcal{SC}}
\newcommand{\hiint}{\xfr}
\newcommand{\co}{\mathrm{conv}}
\begin{document}

\title{The cofinal property of the reflexive indecomposable  Banach spaces}
\author{Spiros A. Argyros and Theocharis Raikoftsalis}
\address{National Technical University of Athens, Faculty of Applied
Sciences, Department of Mathematics, Zografou Campus, 157 80,
Athens, Greece} \email{sargyros@math.ntua.gr, th-raik@hotmail.com}
\footnotetext[1]{2010 \textit{Mathematics Subject Classification}:
46B03,46B06,46B70} \keywords{Banach space theory, $\ell_p$
saturated, Indecomposable spaces, Hereditarily Indecomposable
spaces, Interpolation methods, Saturated norms.}


\begin{abstract} It is shown that every separable reflexive Banach
space is a quotient of a reflexive Hereditarily Indecomposable
space, which yields that every separable reflexive Banach is
isomorphic to a subspace of a reflexive Indecomposable space.
Furthermore, every separable reflexive Banach space is a quotient
of a reflexive complementably $\ell_p$ saturated space with
$1<p<\infty$ and of a $c_0$ saturated space.
\end{abstract}

\maketitle

\section{Introduction}
\label{intr} An infinite dimensional Banach space $X$ is said to
be indecomposable if it is not the topological direct sum of two
infinite dimensional subspaces. In the $70^s$ J. Lindenstrauss
\cite{L2} had asked if every infinite dimensional space is
decomposable. Note that it was already known that the
aforementioned problem has a positive answer for the members of a
variety of classes of Banach spaces. For example, Banach spaces
with an unconditional basis, nonseparable reflexive spaces
\cite{L1} (or more generally nonseparable WCG spaces \cite{AL}),
separable Banach spaces containing $c_0$ \cite{S} are all
decomposable spaces.

On the other hand since 1991 it is known that
Lindenstrauss' problem has an emphatically negative answer. Indeed
W.T. Gowers and B. Maurey's  discovery of Hereditarily
Indecomposable  (HI) spaces (\cite{GM}) has provided examples of
Banach spaces with no decomposable infinite dimensional subspace.
Since the seminal work of Gowers and Maurey the classes of HI and
Indecomposable spaces have been extensively studied leading to
some remarkable results. In particular, new techniques have been
developed concerning the existence of HI spaces having as a
quotient a desired Banach space. These techniques follow two
distinct directions.

The first one, which appeared in \cite{AF}, is closely related to
the DFJP interpolation method (\cite{DFJP}) and makes heavy use of
the geometric aspect of thin sets, which can be traced back to A.
Grothendieck's work (\cite{Gr}) and was explicitly defined in R.
Neidinger's PhD thesis (\cite{N}). This method yielded that every
reflexive space with an unconditional basis has a subspace which
is a quotient of a reflexive HI space. In particular, separable
Hilbert spaces and more generally any $\ell_p$ for $1<p<\infty$
are quotients of reflexive HI spaces. Using duality arguments, one
may also conclude that reflexive $\ell_p$ spaces can be embedded
into a reflexive indecomposable.

The second method is based on saturated and HI extensions of a
ground norming set and led to the most general result concerning
quotients of HI spaces. Namely, as is shown in \cite{ATO}, every
separable Banach space not containing $\ell_1$ is a quotient of a
separable HI space. Comparing the aforementioned techniques, one
should point out that the second leads to more general results but
by its own nature the dual of the resulting HI space is
decomposable. Thus, for a given separable reflexive space $X$ the
corresponding HI space $Y$ which has $X$ as a quotient is never
reflexive. On the contrary, whenever the first method is
applicable it leads to HI spaces with structure similar to the
starting one (i.e. starting with a reflexive space the obtained HI
space remains reflexive).

The aim of the present work is to prove the following:\\

\noindent\textbf{Theorem I.} \textit{Let $X$ be a separable
reflexive space then,
\begin{enumerate}
\item $X$ is a quotient of a reflexive HI space.
\item $X$ is isomorphic to a subspace of a reflexive indecomposable space.
\end{enumerate}} .\\

Since the dual of a reflexive HI is indecomposable $(2)$ is a
direct consequence of $(1)$. The proof of the theorem is based on
a combination of the aforementioned methods and uses certain
auxiliary spaces which are constructed either by interpolation or
extension. More precisely, starting with a reflexive space $X$
with a Schauder basis first we define a space $\auxa$ with a
Schauder tree basis $(e_t)_{t\in\ttt}$, a weakly compact symmetric
subset $W$ of $\auxa$ and a map $\Phi:\auxa\to X$ such that
$\Phi(W)$ is $\frac{1}{8}$ dense in the unit ball of $X$, which
implies that $X$ is a quotient of $\auxa$. The definition of
$\auxa$ shares common features with the corresponding one in
\cite{AF}, however it requires certain modifications as in the
present case the basis of $X$ is not necessarily unconditional. It
is worth mentioning that if we were able to show that $W$ is a
thin or $(a_n)_n$-thin set (c.f. \cite{AF}) of $\auxa$ then
applying a HI interpolation on the pair $(\auxa,W)$ we would
arrive to the reflexive HI space $\xfr$ which has $X$ as a
quotient. This remains unclear and we proceed as follows.

In the second step using DFJP $\ell_2$ interpolation on the pair
$(\auxa, W)$ we obtain a reflexive space $\auxb$ with a Schauder
tree basis $(\tilde{e}_t)_{t\in \ttt_K}$, where $\ttt_{K}$ is a
complete subtree of $\ttt$ and a bounded closed convex set
$\tilde{W}$ such that $\tilde{W}=J_1^{-1}(W)$. Here,
$J_1:\auxb\to\auxa$ is the usual operator mapping the diagonal
space to the original one. Note that the composition operator
$\Phi \circ J_1 $ maps $\tilde{W}$ onto a $\frac{1}{8}$ dense
subset of $B_{X}$ and thus $X$ is a quotient of $\auxb$. As in the
case of $\auxa$ it remains unclear whether the set $\tilde{W}$ is
a thin subset of $\auxb$. Since $\auxb$ is a separable reflexive
space there exists a countable ordinal $\xi$ such that every
weakly null sequence in $\auxb$ does not admit a $\ell^1_{\xi}$
spreading model (c.f. \cite{AGR}).

The next step is the most critical. Here, using a $\xi$-saturated
extension method (\cite{ATO}) we pass to a new space denoted as
$\xfr_{\xi}$ and $I_{\xi}: \xfr_{\xi}\to \auxb$ a bounded linear
injection such that the set $W_{\xi}=I^{-1}_{\xi}(\tilde{W})$ is a
weakly compact and also thin set. Let us note that the structure
of $\xfr_{\xi}$ resembles the generalized Tsirelson space
$T_{\xi}$ (c.f. \cite{ATO}). In that sense $\xfr_{\xi}$ has a much
richer local $\ell_1$ structure than $\auxb$. Thus the thinness is
established in a space with a strong presence of local $\ell_1$
structure which a-priori seems contradictory or at least peculiar.
The final step is the expected one. Namely, we apply a HI
interpolation on the pair $(\xfr_{\xi},W_{\xi})$ to obtain a
diagonal reflexive space $\xfr$ and a bounded convex set
$\tilde{W}_{\xi}$ such that the natural operator $J_{\xi}:\xfr\to
\xfr_{\xi}$ satisfies $J_{\xi}(\tilde{W}_{\xi})=W_{\xi}$. As the
set $W_{\xi}$ is thin the space $\xfr$ is HI and is the desired
one. Indeed the operator $Q=\Phi\circ J_1 \circ I_{\xi}\circ
J_{\xi}$ maps the set $\tilde{W}_{\xi}$ to a $\frac{1}{8}$ dense
subset of $X$ which yields that $Q:\xfr\to X$ is a quotient map.

The paper is organized as follows. Section 2 concerns
preliminaries. Section 3 is devoted to the definition of the space
$\auxa$ which as was mentioned has a Schauder tree basis
$(e_t)_{\ttt}$ equipped with a partial form of unconditionality
defined as "segment-complete unconditional" tree basis. (Def.
\ref{scuncond}). The main result of this section is that when $X$
is reflexive although the space $\auxa$ is not necessarily
reflexive the set $W$ is weakly compact. In section 4 we prove the
following:\\

\noindent\textbf{Theorem II.} \textit{Let $(X,\|\cdot\|)$ be a
reflexive space with a segment complete unconditional tree basis
$(e_t)_{t\in T}$and $\kkk$ be a bounded subset of $X$ such that
for $x\in\kkk$, $\spp x$ is a segment of $T$. Let also
$\xi<\omega_1$ such that $X$ does not admit a $\ell_{\xi}^1$
spreading model. Then, there exists a norm $\|\cdot\|_{\xi}$ on
$c_{00}(T)$ such that setting $\xfr_{\xi}$ to be the completion of
$(c_{00}(T),\|\cdot\|_{\xi})$the following hold:
\begin{enumerate}
\item For every $x\in c_{00}(T)$ $\|x\|\leq \|x\|_{\xi}$. \item
For every $x\in\kkk$, $\|x\|=\|x\|_{\xi}$. \item Denoting by
$W_{\xi}$ the set $\overline{\mathrm{co}}(\kkk\cup - \kkk)$ in
$\xfr_{\xi}$ we obtain that it is weakly compact and thin.
\end{enumerate}}
This theorem provides a tool for constructing thin sets in spaces
with a Schauder tree basis. The norm of the space $\xfr_{\xi}$ is
defined via a norming set $G_{\xi}$ which contains as ground set a
norming subset of the dual of the space $X$ and which is $\sxi$
saturated for finite sequences of functionals with pairwise
incomparable, segment complete supports. In section 5 we show that
the diagonal space in the DFJP $\ell_2$ interpolation applied on
the pair $(\auxa,W)$ has a segment complete unconditional tree
basis. In section 7 combining the results of sections 4,5 and 6 we
prove the following:\\

\noindent\textbf{Theorem III.} \textit{Let $Y$ be a reflexive
space. Then for every $p\in (1,\infty)$ there exists a  reflexive
 space $X_p$ such that every subspace contains an isomorphic copy of $\ell_p$, complemented in the whole space
 and $X_p$ has $Y$ as a quotient. Additionally, $Y$ is a quotient of a separable $c_0$ saturated space.}\\

 It is worth mentioning that the first result in this direction was done by D.H. Leung in
\cite{Le} proving that every separable Hilbert space is a quotient
of a $c_0$ saturated space and it was followed by the results in
\cite{AF} mentioned earlier. More recently, following different
techniques, I. Gasparis in \cite{G1} and \cite{G2} has shown that
certain members of the class of separable reflexive spaces are
quotients of $c_0$ saturated spaces. Let us also point out that
Theorem III provides examples of reflexive Banach spaces with
divergent structure between the spaces and their duals. For
example, there exist spaces $X_p$ as above such that their dual
contains HI subspaces.

In Section 7 we present a variant of the HI interpolation method
appearing in \cite{AF} which is traced to \cite{AD}. The necessity
for modifying the initial HI interpolation method is the
requirement that the diagonal space admits a Schauder basis.
Similarly to \cite{AF} applying the new HI interpolation to a pair
$(X,W)$ with $W$ being convex, symmetric, weakly compact and thin
subset of $X$ we obtain that the diagonal is reflexive HI and this
proves Theorem I in the case where $X$ has a Schauder basis. The
general case of a separable reflexive space mentioned in Theorem I
follows by the classical result of M. Zippin that every separable
reflexive space embeds into a reflexive space with a Schauder
basis (\cite{Z}).

The research included in the present paper was carried out in
2006. In April 2007 Richard Haydon visited us in Athens and with
his collaboration we were able to prove that there exists an
indecomposable space $X$ containing $\ell_1$. After the solution
of the "scalar plus compact" problem (\cite{AH}) the
aforementioned result was adapted to the $\mathcal{L}_{\infty}$
frame as follows:\\

\noindent\textbf{Theorem} \textit{There exists a
$\mathcal{L}_{\infty}$ space $X$ with the scalar plus compact
property containing $\ell_1$.}\\

As it is also mentioned in \cite{AH} the ultimate problem
concerning the cofinal properties of Indecomposable Banach spaces
is the following:\\

\noindent\textbf{Problem} \textit{Does every separable Banach space not containing $c_0$ embed into a separable indecomposable space?}\\

We made an effort to make the present paper as self contained as
possible. Thus, except for a few technical or well known results
all the other proofs are included.

\section{Preliminaries}
\label{prel}
Let us recall some standard notation and definitions for trees.
\begin{notation}
\label{gd2}
\begin{enumerate} \item[ ]
 \item[1.] Let $\Lambda$ be a countable set. By $[\Lambda]^{<\omega}$ and $[\Lambda]$ we denote
 its finite and infinite subsets respectively. We consider $[\Lambda]^{<\omega}$ to be equipped with the partial
 ordering of the initial segment denoted by $\lesseq$.
\item[2.] By a tree on $\Lambda$ we mean a subset $\tr$ of
$[\Lambda]^{<\omega}$ which is backwards closed under $\lesseq$.
\item[3.] Let $\tr$ be a tree on $\Lambda$. A segment of $\tr$ is
a subset of  $\tr$ of the form \{$t\in\tr: t_1\lesseq t \lesseq
t_2$\} with $t_1,t_2\in\tr$. We will usually denote segments of
this form by $[t_1,t_2]$ or more generally by $s$. For $t_1\in
\tr$ we denote by $\hat{t}_1$ the set $\{t\in \tr: t\lesseq
t_1\}$. For a segment $s\subset \tr$, $\hat{s}$ has a similar
meaning, namely $\hat{s}=\{t\in\tr: \exists~ t'\in s\text{ such
that }t\lesseq t'\}$. For $t\in\tr$ we set $|t|$ the cardinality
of the set $\{t'\in\tr:t'\lesseq t\}$ to be the height of $t$. For every
$n\in\nn$ the $n^{th}$-level of $\tr$ is the set \{$t\in\tr:
|t|=n$\}. \item[4.] We identify the branches of a tree $\tr$ with
the elements of the set \{$(a_i)_{i=1}^{\infty}: (a_i)_{i=1}^k\in
\tr$,~$\forall~ k\in\nn$\} and we denote this set by $[\tr]$. For
every $b\in[\tr]$ with $b=(a_i)_{i=1}^{\infty}$ we set
$b|n=(a_1,...,a_n)$. \item[5.] Two nodes $t_1,t_2\in\tr$ are
called comparable if either $t_1\lesseq t_2$ or $t_2\lesseq t_1$.
More generally, if $A_1,A_2\subseteq\tr$ then $A_1,A_2$ are called
comparable if there exist $t_1\in A_1$ and $t_2\in A_2$ which are
comparable. Otherwise $A_1, A_2$ are called incomparable. We will
write $A_1\perp A_2$ to denote the fact that $A_1, A_2$ are incomparable.
\item[6.] For $t\in\tr$ by $\mathcal{O}_t$
we denote the set $\mathcal{O}_t=$ \{$b\in [\tr]: \exists n\in\nn$
such that $b|n=t$\}. The sets \{$\mathcal{O}_t: t\in\tr$\} form
the usual basis of the topology of $[\tr]$. \item[6.] For every
$t\in\tr$ and $b\in[\tr]$ we will write $t\in b$ if $\exists$
$n\in\nn$ such that $b|n=t$. For every segment $s$ of $\tr$ and
$b\in\tr$ we will write $s\subseteq b$ if $\forall t\in s$ it
holds $t\in b$.
\end{enumerate}
\end{notation}
In the following sections all trees are considered countable, finitely splitting and with nonempty bodies.

\begin{defn}
\label{bijection}
For every such tree $\tr$ we fix a bijection $h_{\tr}: \tr\mapsto\nn$ such
that the following hold:

\begin{enumerate} \item[i.] $h_{\tr}(t_1)<h_{\tr}(t_2)$, whenever $|t_1|<|t_2|$
\item[ii.] If $t_1,t_2\in\tr$ and $|t_1|=|t_2|$ i.e.
$t_1=(a_1,...,a_n)$ and $t_2=(b_1,...,b_n)$ then $h_{\tr}(t_1)<h_{\tr}(t_2)$,
whenever $a_n<b_n$.
\end{enumerate}
\end{defn}
When the tree $\tr$ is understood we will refer to $h_{\tr}$ simply as $h$.
We denote by $c_{00}(\tr)$ the linear space of all functions
$f:\tr\mapsto\rr$ such that $\spp(f)= \{ t\in\tr: f(t)\neq 0\}$ is a finite set. We also
denote by $(e_t)_{t\in\tr}$ the standard Hamel basis of  $c_{00}(\tr)$
consisting of the characteristic functions of all singletons
$\{t\}\subseteq\tr$.

\begin{defn}
\label{gd3} Let $(A_i)_{i=1}^{\infty}$ be a sequence of finite
subsets of $\tr$. We will say that $(A_i)_{i=1}^{\infty}$ is
\begin{enumerate}
\item[i.] a block sequence if $\max$\{$h(t):t\in A_i$\}$<
\min$\{$h(t):t\in A_{i+1}$\} and we will write
$A_1<A_2<...<A_n<...$ \item[ii.] a level-block sequence if
$\max$\{$|t|:t\in A_i$\}$< \min$\{$|(t|:t\in A_{i+1}$\} and we
will write $A_1\lblo A_2\lblo ...\lblo A_n\lblo ...$ \item[iii.]
For a sequence $(f_i)_{i=1}^{\infty}$ of elements of $c_{00}(\tr)$ we will
say that $(f_i)_{i=1}^{\infty}$ is block (level-block)  if $(\spp
f_i)_{i=1}^{\infty}$ is a block (level-block, respectively)
sequence of subsets of $\tr$.
\end{enumerate}
\end{defn}

For the sake of simplicity of notation if $A$ is a subset of
$\tr$ we will write $\min A$ for $\min$\{$h(t):t\in A$\} and if
$A$ is finite $\max A$ for $\max$\{$h(t):t\in A$\}. We will also
write $\min^l A$ for $=\min$\{$|t|:t\in A$\} and $\max^l A$ for
$\max\{t:\text{ }t\in A\}$ .
\begin{defn}
\label{gd4} Let $A,I$ be  subsets of $\tr$.
\begin{enumerate}
\item[i.]  We will call $A$ segment complete if for every
$t_1,t_2$ in A, with $t_1\lesseq t_2$ \\
$[t_1,t_2]\subseteq A$. \item[ii.] $I$ will be called an interval
of  $ \tr$  if $h(I)=$\{$h(t):t\in\tr$\} is an interval of
$\nn$.
\end{enumerate}
\end{defn}
We note that every interval $I$ of $\text{}\tr$ is segment
complete.
\begin{defn}
\label{gd5} For every $f\in$ $c_{00}(\tr)$ we define $\rg f$ to be the
minimal interval $I$ of $\tr$ such that $\spp f \subseteq I$.
Similarly, we set $\rg^l f$ to be the minimal interval $\I^l$ of
$\tr$ of the form $I^l=$\{$t\in\tr: m\leq |t|\leq M$\} such that
$\spp f \subseteq I^l$.
\end{defn}
\begin{rem}
\label{gr2}
 It is clear that a sequence $(f_i)_{i=1}^{\infty}$ in
$c_{00}(\tr)$ is
\begin{enumerate}
\item[i.] block if $\rg f_i<\rg f_{i+1}$ $\forall i\in\nn$
\item[ii.]level-block if $\rg^l f_i<\rg^l f_{i+1}$ $\forall
i\in\nn$.
\end{enumerate}
\end{rem}

\section{Tree representation of the ball of a Banach space}
\label{tree}
Let $X$ be an arbitrary Banach space with a bimonotone, normalized
Schauder basis $(x_i)_{i\in\nn}$ and $(x_i^*)_{i\in\nn}$ the
biorthogonal functionals of $(x_i)_{i\in\nn}$ in $X^*$. In this
section we define a tree $\ttt$ and a norm on $c_{00}(\ttt)$ that will help us "spread" along
the branches of $\ttt$ a set $\mathrm{K}$ which is isometric (via
a map $\Phi$ to be defined later) to a $\frac{1}{8}$-net in the unit ball of
$X$. We will denote by $\auxa$ the completion of $c_{00}(\ttt)$ with respect to this
norm. This technique gives $X$ as a quotient of $\auxa$. In
addition we show that if $X$ is reflexive then the set
$\mathrm{K}$ is weakly compact in $\auxa$. We start with the
following  general definition.

\begin{defn}
\label{scuncond} Let $T$ be a tree and a norm $\|\cdot\|$ defined
on $c_{00}(T)$ such that the sequence $(e_t)_{t\in T}$ is a Schauder basis for the completion of $c_{00}(T)$ denoted by $X_T$. Then
\begin{enumerate}
\item[1.] The norm $\|\cdot\|$ and the basis $(e_t)_{t\in T}$ will
be called $\mathcal{SC}$-unconditional if for every $A\subseteq T$
segment complete and $x=\sum_{t\in T}\lambda_t e_t\in X_T$ we
have:
\[\|\sum_{t\in A}\lambda_t e_t\|\leq \|\sum_{t\in T} \lambda_t
e_t\|\] \item[2.] Let $\psi: T\to [-1,1]$ be a function assigning
to each node of $T$ a scalar $\psi(t)\in [-1,1]$. Let also $C>0$.
We consider for each $t\in T$ the vector $y_t=\sum_{t'\lesseq
t}\psi(t')e_{t'}$ and set $\kkk_{\psi}^1=\{y_t\in X_T: \|y_t\|\leq
C\}$, $\kkk_{\psi}^2=\{y_t|I: \|y_t\|\leq C\text{ and }I \text{ is
an interval of }T\}$ and
$\kkk_{\psi}=\overline{\kkk_{\psi}^2}^{\|\cdot\|}$
\end{enumerate}
\end{defn}

For the rest of this section we assume that $X$ is a fixed Banach
space with a Schauder basis $(x_i)_i$ and $(x^*_i)_i$ the
biorthogonal functionals in $X^*$. We pass on to define a
$\scu$-unconditional norm on $c_{00}(\ttt)$ where $\ttt$ is an
appropriately defined tree such that the completion of this space
with respect to this norm has a quotient isomorphic to $X$. We start with the
definition of $\ttt$.

\begin{defn}
\label{gd1} Let $(F_n)_{n=1}^\infty$ be the following sequence in
$([-1,1]\cap\qq)^{<\omega}$
\begin{eqnarray*}
F_n&=&\{\pm\frac{i}{8^n}:1\leq i\leq 8^n\} \text{ }\forall
n\in\nn.
\end{eqnarray*}
We set
\begin{eqnarray*}
\ttt &=&\{(a_1,a_2,...,a_k) : a_i\in F_i ,i\leq k, k\in\nn\}
\end{eqnarray*}
\end{defn}

It can be readily seen that $\ttt$ is a countable, finitely splitting tree such that every $t\in\ttt$ with $|t|=n$, has $2\cdot 8^{n+1}$ immediate successors.\\

The norming set $\go(X)$ of $\auxa$ is defined as follows:
\begin{defn}
Let $\mathrm{G}_0^1(X)$ be the following subset of $c_{00}(\ttt)$
\begin{eqnarray*}
\mathrm{G}_0^1(X)& = & \big{\{}\sum_{i=1}^n a_i \sum_{|t|=i}
e_t^*: \|\sum_{i=1}^n a_i x_i^* \|_{X^*} \leq 1 \big{\}}
\end{eqnarray*}
Set
 \begin{center}$\go(X)=$ \{$f|_A: f\in
\mathrm{G}_0^1(X)\text{ and }A \text{  is a segment complete
subset of  } \ttt$ \}\end{center} where $f|_A$ denotes the
restriction of $f$ on $A$.
\end{defn}

\bigskip

We consider the norm on $c_{00}(\ttt)$ induced by the set $\go(X)$. Namely,
\begin{center}$\forall x\in$ $c_{00}(\ttt)$ we set $\|x\|=\sup$
\{$f(x):f\in\go$\}\end{center} The space $\auxa$ is the
completion of $c_{00}(\ttt)$ under the norm defined above. It can be readily
verified that the sequence $(e_t)_{t\in\ttt}$ (enumerated through
$h$) becomes a bimonotone, normalized Schauder basis of $\auxa$. We
also have the following easy observation:
\begin{rem}
\label{segcom1} For every $A\subseteq\ttt$ segment complete the
natural projection $P_A:\auxa\mapsto\auxa$ defined by
$P_A(\sum_{t\in\ttt} \lambda_t e_t)= \sum_{t\in A} \lambda_t e_t$
has norm one.
\end{rem}
We will need the following Lemma that gives a description of the
pointwise closure of the norming set $\go(X)$.
\begin{lem}
\label{gopointiwise} Let
\begin{equation*}
G=\big{\{}f|A: f=\sum_{i=1}^{\infty} a_i \sum_{|t|=i} e^*_t ,
~\|\sum_{i=1}^{\infty}a_i x^*_i\|_{X^*}\leq 1 \text{ and }A \text{
is a segment complete subset of } \ttt \big{\}}
\end{equation*}
where all limits are taken with respect to the $w^*$-topology. Then $\overline{\go (X)}^p=G$.
\end{lem}
\begin{proof}
It is easy to see that $G\subseteq \overline{\go (X)}^p$. Let $f\in\overline{\go(X)}^p$. Then there exists a sequence $(f_n)_n$ in $\go(X)$ such that $f_n\stackrel{p}\to f$. Each
$f_n$ is of the form $f_n=\mathcal{X}_{A_n}\cdot(\sum_{i=1}^{k_n}
a^n_i\sum_{|t|=i} e^*_t)$, where $A_n$ are finite segment complete
subsets of $\ttt$, $\mathcal{X}_{A_n}$ is the characteristic
function of $A_n$ and $\|\sum_{i=1}^{k_n} a^n_i x^*_i\|\leq 1$. Let $g_n=\sum_{i=1}^{k_n} a^n_i x^*_i$.
Then there exists a $M\in [\nn]$ such that the sequence $(g_n)_{n\in M}$ converges $w*$ to a $g\in B_{X^*}$. Let $i\in \nn$. Denote by $a_i$ the limit $\lim_{n\in M}a_i^n$ and set $A=\liminf_{n\in M} A_n$ which can be readily seen to be a segment complete subset of $\ttt$. Considering the functional $\hat{f}=\mathcal{X}_{A}\sum_{i=1}^{\infty} a_i\sum_{|t|=i}e^*_t$, we claim that the sequence $(f_n)_{n\in M}$ converges $w^*$ to $\hat{f}$. Indeed, let $t\in\ttt$ with $|t|=i$. If $t\in A$, then $\lim_{n\in M} f_n(e_t)=a_i=\hat{f}(e_t)$. Assuming that $t\notin A$ then $\lim_{n\in M}f_n(e_t)=\hat{f}(e_t)=0$. Hence, $\hat{f}=f\in G$ and the proof is complete.
\end{proof}

We pass on to define the map $\psi:\ttt\to\nn$ that will give us the corresponding set $\kkk_\psi$
\subsection {The set $\mathbf{K}$.\\}

 The set $\kkk$ and the map $\Phi$ are defined as follows:

\begin{defn}
\label{thesetkkk}
Let $\psi:\ttt:\to\nn$ be the following assignment. For every $t=(a_1,a_2,...,a_n)\in\ttt$ we set $\psi(t)=a_t=a_n$ and
$y_t=\sum_{t'\lesseq t} a_{t'}e_{t'}$.
We set $\kkk=\kkk_{\psi}$ as in Definition \ref{scuncond}.
\end{defn}

\begin{defn} \label{themapf}
We consider a map $\Phi:\auxa\mapsto X$ defined as
$\Phi(\sum_{t\in\ttt}\lambda_te_t)=\sum_{i=1}^{\infty}(\sum_{|t|=i}\lambda_t)x_i$.
\end{defn}

\begin{rem}
\label{prp1} We can see that $\Phi$ is a bounded linear operator
with $\|\Phi \|\leq 1$. In addition, for  $b\in[\ttt]$ if we denote
by $X_b$ the subspace $X_b=\overline{<e_t:t\in b>}^{\|\cdot\|}$
then we have that $\Phi$ restricted to $X_b$ is an isometry.
\end{rem}
We also have the following
\begin{prop}
\label{dense} The set $\Phi(\kkk)$ is a $\frac{1}{8}$-net in
the unit ball $B_X$ of $X$. Moreover, the map $\Phi$ is onto.
\end{prop}
\begin{proof}
Let $x=\sum_n b_n x_n \in B_X$. Since the basis of $X$ is
bimonotone we have that $|b_n| \leq 1$, for all $n\in\nn$. For
each $n\in\nn$ we can choose $a_n$ in the set
$F_n=\{\pm\frac{i}{8^n}:1\leq i\leq  8^n\}$ such that $|b_n - a_n|
\leq\frac{1}{8^n}$. If we set $\sigma=(a_n)_{n=1}^{\infty}\in
[\ttt]$ then $y_{\sigma}=\sum_n a_n e_{\sigma|n}\in \kkk$ and
$\|\Phi(y_{\sigma})-x\|_{X}\leq \frac{1}{8}$.  Indeed, if we let
$X_{\sigma}=\overline{<e_t:t\in \sigma>}$, then by Remark
\ref{prp1} $\Phi: X_{\sigma} \mapsto X$ is an isometry. Thus if we
denote the restriction of $\Phi$ on $X_{\sigma}$ by
$\Phi_{\sigma}$ we have that $\Phi_{\sigma}^{-1}(\sum_n b_n x_n)=
\sum_n b_n e_{\sigma|n}$ and

$\|\sum_n b_n e_{\sigma|n}\|_{\auxa} = \|\Phi(\sum_n b_n e_{\sigma|n})\|_X =
\|\sum_n b_n x_n\|_X \leq 1$. Thus,
\begin{equation*}
\|\sum_n a_n e_{\sigma|n}\|_{\auxa}  \leq  \|\sum b_n e_{\sigma|n}\|_{\auxa}
  +  \sum_n |a_n - b_n|\leq 1+ \frac{1}{8}
 \end{equation*}
 This gives us that $y_{\sigma}$ is an element of $\kkk$. Finally,
 \begin{equation*}
 \|\sum_n b_n e_{\sigma|n} - \sum_n a_n
 e_{\sigma|n}\|_{\auxa} = \|\Phi(\sum_n b_n e_{\sigma|n}) - \Phi(\sum_n a_n
 e_{\sigma|n})\|_X = \|\Phi(y_{\sigma}) - x\|_X \leq \frac{1}{8}
\end{equation*}
\end{proof}
The following Lemma shows the behavior of incomparably supported
sequences of vectors in $\kkk$ if we assume that $X$ has a
shrinking basis.
 \begin{lem}
\label{incomparablek} Suppose $(x_i)_i$ is a shrinking basis for
$X$, then for every sequence $ (y_n)_n$ in $\kkk$ such that $(\spp
y_n)_n$ are finite and mutually incomparable subsets of $\ttt$ we
have that $y_n\stackrel{w}\to 0$.
\end{lem}
\begin{proof}
Let $(y_n)_n$ be as above. In order to prove that
$(y_n)_n$ is weakly null it is enough to show that $f(y_n)\to 0,
\forall f\in\overline{\go}^p$. Choose $f\in\overline{\go}^p$. By Lemma \ref{gopointiwise} there exist a $g=\sum_{i=1}^{\infty} b_i x_i^*\in B_{X^*}$ and a segment complete $A\subset\ttt$ so that $f=\mathcal{X}_A\sum_{i=1}^{\infty}b_i\sum_{|t|=i}e^*_t$. Let $s^A_n = \spp y_n\cap A$, $z_n=y_n|_A=\sum_{t\in s_n^A}a_te_t$ and observe the following,

\[f(y_n)=\sum_{i=1}^{\infty}b_i\sum_{|t|=i}e^*_t(z_n)=g(\Phi(z_n))\].

Hence, it is enough to show that $g(\Phi(z_n))\to 0$. As $\Phi(z_n)=\sum_{t\in s_n^A}a_t x_{|t|}$ it can be seen that for $i\in\nn$, $x^*_i(\Phi(z_n))=a_t$, if $s_n^A\cap L_i\neq\emptyset$ and zero otherwise. Since $\ttt$ is finitely branching and $s_n^A\perp s_m^A$ for all $n\neq m\in\nn$, we deduce that for a fixed $i\in\nn$ the set $\{n\in\nn: s_n^A\cap L_i\neq\emptyset\}$ is finite. Thus, $x^*_i(z_n)\to 0$. Finally, as $(x_i)_i$ is shrinking we obtain that $g(z_n)\to 0$ and this completes the proof.
\end{proof}

\begin{prop} \label{prp2} For every
reflexive Banach space $X$ with a bimonotone normalized Schauder
basis $(x_i)_i$ the set $\kkk$ is weakly compact.
\end{prop}
\begin{proof}
Let $(w_n)_n$ be a sequence in $\mathrm{K}$. Up to an arbitrarily small perturbation we may assume
that $\spp w_n$ is finite for all $n\in\nn$. We set $s_n= \spp
w_n$ and we observe that each $s_n$ is a finite segment of $\ttt$.
So, $w_n=\sum_{t\in s_n}a_t e_t$. We may also assume (by passing
to a subsequence if needed) that for each $t\in\ttt$,
$e_t^*(w_n)\to w(t)\in\rr$. We set $S=$\{$t\in\ttt: w(t)\neq 0$\}.
We know that $S$ is a segment of $\ttt$ (finite or infinite).
Thus, we may assume that each $y_n$ has a decomposition $y_n= u_n
+ y_n$ where
\begin{enumerate}
\item[i.] $(\spp u_n)_n$ is an increasing (with respect to
$\subseteq$) sequence of segments of $S$ \item[ii.] $(\spp y_n)_n$
is a sequence of incomparable segments of $\ttt$.\end{enumerate}
We set $w_s=\sum_{t\in S}a_te_t$ if $\mathcal{S}\neq\emptyset$ and
$0$ otherwise. We claim that $w_n\stackrel{w}{\to} w_s$. Lemma
\ref{incomparablek} yields that $(y_n)_n$ is weakly null. To
finish the proof we shall show that $u_n \stackrel{\|\cdot\|}\to
w_s$. Indeed, for every $n\in\nn$ we have that $\spp u_n\lesseq S$
and since $w_n\in\kkk$ we have that $u_n\in\kkk$ for all
$n\in\nn$. Thus, as the basis of $X$ is boundedly complete, $u_n\stackrel{p}\to w_s$ and $\|\Phi(u_n)\|_{X}= \|u_n\|_{\auxa}$
we have that $w_s\in\kkk$ and $u_n\stackrel{\|\cdot\|}\to w_s$.
\end{proof}
\begin{rem}
The connection between the set $\kkk$ and $B_X$ when $X$ is an
arbitrary Banach space with a basis is not completely clear to us.
For example, if one considers $X$ to be $c_0$ with the summing
basis then it turns out that $\kkk$ contains a sequence equivalent
to the standard $\ell^1$ basis. Indeed, we notice that the norming
set in this case becomes $\overline{\go}^p=\{\pm\mathcal{X}_A: A\subseteq\ttt
\text{ segment complete}\}$ where by $\mathcal{X}_A$ we denote the
characteristic function of $A$. We can construct a sequence
$(w_n)_n$ in $\kkk$ which has no weakly Cauchy subsequence. To see this choose a sequence $(t_n)_n\subset \ttt$ with the following properties,
\begin{enumerate}
\item[i.] Each $t_n$ is of the form $t_n=(a_1,...,a_{k_n})$ and $a_{k_n}=a_{t_n}=\frac{1}{2}$;
\item[ii.] $t_n\perp t_m$ for all $n\neq m\in\nn$.
\end{enumerate}
For each $n\in\nn$ set $t'_n=t_n\smallfrown\frac{-1}{2}=(a_1,...,a_{k_n},\frac{-1}{2})$ and
 $w_n=a_{t_n}e_{t_n}+a_{t'_n}e_{t'_n}$. To see that $(w_n)_n$ satisfies the desired property, choose a subsequence$(w_{m_i})_{i\in\nn}$. Let $t(i)=t_{m_i}$ when $i=2k$ and
$t(i)=t'_{m_i}$ when $i=2k-1$. It is
clear that $(t(i))_i$ are mutually incomparable. Therefore, the set
$A=\cup_{i=1}^{\infty} \{t(i)\}$ is segment complete. So the functional
$f=\mathcal{X}_A\in\overline{\go}^p$ estimates
$|f(w_{m_i})-f(w_{m_{i+1}})|=1$ for all $i$. By Rosenthal's
$\ell^1$ theorem \cite{Ro1} we obtain that $(w_n)_n$ is equivalent to
the $\ell^1$ basis.
\end{rem}

\section{Thin subsets of Banach spaces}
\label{thin} Let $T$ be a tree and $\|\cdot\|_{X_T}$ be a
$\mathcal{SC}$-unconditional norm defined on $c_{00}(T)$. Denote
by $X_T$ the completion of $c_{00}(T)$ with respect to
$\|\cdot\|_{X_T}$.  Fix also a function $\psi: T\to
[-1,1]$ and $\kkk_{\psi}$ (referred to as $\kkk$ for simplicity) as in Definition \ref{scuncond}. \\
In this section we present a general method for extending the norm
of $X_T$ to a new norm defined on $c_{00}(T)$ such that the
completion of this space contains $\kkk$ as a \emph{thin} subset.
Namely, the entire section is devoted to the proof of the
following theorem:
\begin{thm}
\label{bthinthm} Suppose that $X_T$ is reflexive and $\kkk$ is a weakly compact subset of $X_T$.
Then there exists a space $\auxx$ such that the following hold:
\begin{enumerate}
\item[1.] The identity operator $I:\auxx\to X_T$ is continuous.
\item[2.] $\kkk\subset I(\auxx)$ and the closed convex hull of $(I^{-1}(\kkk)\cup I^{-1}(- \kkk))$ is s weakly compact and thin subset of $\auxx$.
\end{enumerate}
\end{thm}
The notion of thinness was introduced in \cite{N2} and was
extensively used in \cite{AF} where several methods for proving
that a set satisfies this property were developed. We give the corresponding definition in
subsection \ref{thinness} where we also prove the aforementioned Theorem. Before doing so though we need some preparatory work which is done in the following subsection.

\subsection{Tsirelson type spaces and norms.\\}

We start with some preliminary results concerning
families of finite subsets of $\nn$ and Tsirelson type norms. Most
of these results are well known and have been extensively used in
the relevant literature, with the exception of Lemmas \ref{ltsi1},
\ref{lth1}, \ref{lth2} and Remark \ref{rtsi1} which can be found
in \cite{LM} and were brought to our attention by the authors. We
include this subsection in order to make the text as
self-contained as possible. We start by recalling the following notions concerning families of finite
subsets of $\nn$.
\begin{defn}
\label{dtsi1} Let $\mathcal{M}$ be a family of finite subsets of
$\nn$. $\mathcal{M}$ is called
\begin{enumerate}
\item[i.] Compact if the set of characteristic functions
$\{\mathcal{X}_A:A\in\mathcal{M}\}$ is a compact subset of
$\{0,1\}^{\nn}$ \item[ii.] Hereditary if for every
$A\in\mathcal{M}$ and $B\subseteq A$ we have $B\in\mathcal{M}$
\item[iii.] Spreading if for every
$A=\{t_1<t_2<...<t_r\}\in\mathcal{M}$ and
$B=\{t_1'<t_2'<...<t_r'\}$ with $t_i\leq t_i' \text{ } \forall
i=1,...,r$ we have $B\in\mathcal{M}$.
\end{enumerate}
\end{defn}
\begin{defn}
\label{dtsi2} Let $\mathcal{M}\subseteq [\nn]^{<\omega}$.
\begin{enumerate}
\item[i.] A finite sequence $(E_1,...,E_n)$ of successive and
finite subsets of $\nn$ is called $\mathcal{M}-\text{admissible}$
if there exists $F\in\mathcal{M}$ with $F=\{m_1<m_2<...<m_n\}$
such that $m_1\leq E_1<m_2\leq E_3<...<m_n\leq E_n$ \item[ii.]A
finite sequence $(f_1,...,f_n)$ of vectors in $c_{00}(\nn)$ is
called $\mathcal{M}-\text{admissible}$ if $(\spp f_i)_{i=1}^n$ is
$\mathcal{M}-\text{admissible}$.
\end{enumerate}
\end{defn}
\begin{defn}
\label{dtsi3} Let $\mathcal{F},\mathcal{G}$ be two families of
finite subsets of $\nn$ we define:
\begin{enumerate}
\item[i.] The block sum $\mathcal{F}\oplus\mathcal{G}=\{M\cup N:
M<N, M\in\mathcal{G},N\in\mathcal{F}\}$ \item[ii.] The convolution
$\mathcal{F}\otimes\mathcal{G}=\{\cup_{i=1}^n F_i:F_1<...<F_n, F_i
\in \mathcal{F}, i=1,...,n \text{and } \{\min
F_i\}_{i=1}^n\in\mathcal{G}\}$
\end{enumerate}
\end{defn}
\begin{defn}
\label{dtsi4} The Schreier hierarchy was first defined in
\cite{AA}. It is a set of families $(\sxi)_{\xi<\omega_1}$ of
finite subsets of $\nn$ which can be defined recursively as
follows:
\begin{center}
$\mathcal{S}_0=\{\{t\}:t\in\nn\}\cup\{\emptyset\}$
\end{center}
Let $\couo$ and suppose that $\sxi$ have been defined for all
$\zeta<\xi$ Then
\begin{enumerate}
\item[i.] If $\xi=\zeta+1$ we set
$\sxi=\mathcal{S}_{\zeta}\otimes\mathcal{S}_1$ \item[ii.] If $\xi$
is a limit ordinal then we fix a strictly increasing sequence of
non-limit ordinals $(\xi_n)_n$ with $\sup\xi_n=\xi$ and set
$\sxi=\bigcup_{n=1}^{\infty}\{F\in\mathcal{S}_{\xi_n}:F\geq n\}$
\end{enumerate}
\end{defn}
It can be verified by transfinite induction that each $\sxi$ for
$\couo$ is compact hereditary and spreading. We will need the
following two results found in \cite{LM} concerning the families
$(\sxi)_{\couo}$ which can be proved by transfinite induction.
\begin{lem}
\label{ltsi1} For every ordinal $\couo$ and $M\in[\nn]$ we have
\begin{enumerate}
\item[i.] $[M]^{\leq 3}\otimes\sxi\subseteq\sxi\otimes [M]^{\leq
2}$ \item[ii.] If $\min M\geq3$, then $[M]^{\leq 3}\otimes
(\sxi\oplus [M]^{\leq 1})\subseteq\sxi\otimes [M]^{\leq 3}$
\end{enumerate}
\end{lem}
\begin{rem}
\label{rtsi1} By Lemma \ref{ltsi1} we have that $\forall
M\in[\nn]$ and $\couo$ \begin{center}$[M]^{\leq 8}$
$\otimes(\sxi\otimes [M]^{\leq 2})\subseteq ([M]^{\leq 3}\otimes
([M]^{\leq 3}\otimes\sxi))\otimes [M]^{\leq 2}\subseteq
(\sxi\otimes [M]^{\leq 4})\otimes [M]^{\leq 2}\subseteq\sxi\otimes
[M]^{\leq 8}$.\end{center}
\end{rem}
\noindent{\textbf{The Spaces} $T(\theta,\mathcal{F})$\\
Let $0<\theta<1$ and $\mathcal{F}$ be a compact hereditary family
of finite subsets of $\nn$.
\begin{defn}
\label{dth1} Let $G_{\theta,\mathcal{F}}$ be the minimal subset of
$c_{00}(\nn)$ such that
\begin{enumerate}
\item[i.] $\pm e_n\in$\gft$\forall n\in\nn$ \item[ii.]\gft is
closed under the $(\theta,\mathcal{F})$-operation. That is if
$(f_i)_{i=1}^d$ is an $\mathcal{F}$-admissible family in \con then
$\theta\sum_{i=1}^d f_i\in$ \gft.
\end{enumerate}
The space $T(\theta, \mathcal{F})$ is the completion of \con under
the following norm \begin{center}$\forall x\in$ \con we set
$\|x\|_{(\theta, \mathcal{F})}=\sup\{f(x): f\in
G_{\theta,\mathcal{F}}\}$.\end{center} \end{defn} Detailed
expositions of the $T(\theta,\mathcal{F})$ type spaces can be
found in \cite{AT}. In the sequel we shall denote
$T_{\xi}=T(\frac{1}{2}, \sxi)$, $T_{\xi}^1=T(\frac{1}{2},
\sxi\otimes [\nn]^{\leq 2})$ and $T_{\xi}^2=T(\frac{1}{2},
\sxi\oplus [\nn]^{\leq 1})$. The following two Lemmas are results
in \cite{LM} but for the sake of completion we include their
proofs here.
\begin{lem}
\label{lth1} Let $\couo$. Then for every finite sequence
$(b_i)_{i=1}^k$ of scalars we have \begin{enumerate} \item[I.]
$\|\sum_{i=1}^k b_i e_i\|_{T_{\xi}^1}\leq 8\|\sum_{i=1}^k b_i
e_i\|_{T_{\xi}}$ \item[II.]$\|\sum_{i=1}^k b_i
e_i\|_{T_{\xi}^2}\leq 3\|\sum_{i=1}^k b_i e_i\|_{T_{\xi}}$
\end{enumerate}
Where by $(e_i)_{i\in\nn}$ we denote the standard Hamel basis of
\con.
\end{lem}
\begin{proof}
I. By Remark \ref{rtsi1} we have that $[\nn]^{\leq 8}\otimes
(\sxi\otimes [\nn]^{\leq 2})\subseteq \sxi\otimes [\nn]^{\leq 8}$.
Let  $G_{T_{\xi}^1}$ be the norming set of ${T_{\xi}^1}$ and let
$f\in\ G_{T_{\xi}^1}$. We will define $g_1<...<g_l$ with $l\leq 8$
and $g_i\in\ G_{T_{\xi}}, i=1,...,l$ such that $f=\sum_{i=1}^l
g_i$. We use induction on the complexity of $f$. Let $f=\pm e_n$
for some $n\in\nn$,then there is nothing to prove. Let
$f=\frac{1}{2}\sum_{i=1}^d$ such that
\begin{enumerate} \item[i.] $f_1<...<f_d$ \item[ii.]
$(f_i)_{i=1}^d$ is an $\sxi\otimes [\nn]^{\leq 2}$ admissible
sequence \item[iii.] For every $f_i$ there exists a sequence
$g_1^i<...<g_{l_i}^i$ such that $g_j^i\in G_{T_{\xi}}$ for
$j=1,...,l_i$ and $f_i=\sum_{j=1}^{l_i} g_j^i$
\end{enumerate}
Now since $\{\min f_i\}_{i=1}^d\in\mathcal{S}_{\xi}\otimes
[\nn]^{\leq 2}$ by Remark \ref{rtsi1} that $\bigcup_{i=1}^d\{\min
g_j^i:j\leq l_i\}\in\mathcal{S}_{\xi}\otimes [\nn]^{\leq 8}$. Thus
there exist $B_1,...,B_k$ with $k\leq 8$ such that
\begin{enumerate}\item[1.] $B_m\in\sxi$ for all $m=1,...,k$ and
\item[2.]$B_1<...<B_k$ \end{enumerate}such that
$\bigcup_{i=1}^d\{\min g_j^i:j\leq l_i\}=\bigcup_{m=1}^k B_m$. By
setting $g^{(m)}=\frac{1}{2}\big{(}\sum_{\min g_j^i\in
B_m}g_i^j\big{)}$ we get \begin{enumerate} \item[a.]
$f=\sum_{m=1}^k g^{(m)}$ and \item[b.] $g^{(m)}\in
G_{T_{\xi}}$\end{enumerate} as desired.\\
II. By Lemma \ref{ltsi1} we have $[\nn]^{\leq 3}\otimes
(\sxi\oplus [\nn]^{\leq 1})\subseteq\sxi\otimes [\nn]^{\leq 3}$.
Using the same arguments as in the proof of I. we conclude that
for every $f\in G_{T_{\xi}^2}$ there exist $g_1<g_2<g_3$ with
$g_i\in G_{T_{\xi}}$ for i=1,2,3 such that $f=\sum_{i=1}^3 g_i$.
\end{proof}
\begin{defn}
\label{dth2} Let $M\in[\nn]$ with
$M=\{m_1<m_2<...\}$.\begin{enumerate}\item[i.] For every $m\in M$
we set $m^{+}$ to be the immediate successor of $m$ in $M$, that
is $m_i^{+}=m_{i+1}$ for all $i\in\nn$\item[ii.] If $A\in
[M]^{\leq\omega}$ then we set $A^{+}=\{m^{+}:m\in A\}$
\end{enumerate}
\end{defn}
\begin{lem}
\label{lth2} Let $\couo$ and $M\in[\nn]$ with $M=\{m_1<m_2<...\}$.
Then for every finite sequence $(b_i)_{i=1}^k$ of scalars we have
$\|\sum_{i=1}^k b_i e_{m_i}\|\leq\|\sum_{i=1}^k b_i
e_{m_i^{+}}\|\leq 3\|\sum_{i=1}^k b_i e_{m_i}\|$ where all norms
are considered in the space $T_{\xi}$.
\end{lem}
\begin{proof}
Let $A\in [M]^{\leq\omega}$. Suppose that $A^{+}=\{m^{+}:m\in
A\}\in\sxi$. Since $\sxi$ is a spreading family it follows that
$A\setminus \min A\in\sxi$. Thus, $A\in\sxi\oplus [M]^{\leq 1}$.
So, if we consider $f\in G_{T_{\xi}}$ with the property $\spp
f\subseteq \{m_i^{+}:i=1,...,k\}$ there is an $f'\in
G_{T_{\xi}^2}$ such that $f(\sum_{i=1}^k b_i e_{m_i^{+}})\leq
f'(\sum_{i=1}^k b_i e_{m_i})$ and this gives
\begin{enumerate}\item[i.]$\|\sum_{i=1}^k b_i
e_{m_i^{+}}\|_{T_{\xi}}\leq\|\sum_{i=1}^k b_i
e_{m_i}\|_{T_{\xi}^2}$.\end{enumerate} On the other hand since
$\sxi$ is spreading it is easily verified that
\begin {enumerate}\item[ii.]$\|\sum_{i=1}^k b_i e_{m_i}\|_{T_{\xi}}\leq\|\sum_{i=1}^k b_i
e_{m_i^{+}}\|_{T_{\xi}}$\end{enumerate} combining i. and ii. we
have \begin{center}$\|\sum_{i=1}^k b_i
e_{m_i}\|_{T_{\xi}}\leq\|\sum_{i=1}^k b_i
e_{m_i^{+}}\|_{T_{\xi}}\leq\|\sum_{i=1}^k b_i
e_{m_i}\|_{T_{\xi}^2}$\end{center} and by Lemma \ref{lth1} we get
the desired.
\end{proof}
\subsection{The norming set $G_{\xi}$.\\}

In this subsection starting with $X_T$, as in the introductory
paragraph of the section, assuming that $X_T$ does not contain an
isomorphic copy of $\ell^1$ we define $\auxx$ and prove that it
satisfies the first two properties of $Y$ mentioned in Theorem
\ref{bthinthm}. Namely, we show that the identity map $I:\auxx\to
X_T$ is continuous and that the set $I^{-1}(\kkk)$ is a weakly
compact subset of $\auxx$. We start with some well known results
concerning $\ell^1$-spreading models.
\begin{defn}
 \label{dxi1} A bounded sequence $(x_n)_n$ in a Banach space $Y$
is an $\ell_{\xi}^1$-spreading model, for $\couo$, if there exists
a constant $C>0$ such that \begin{center} $\|\sum_{i\in F}a_i
x_i\|\geq C\sum_{i\in F}|a_i|$ \end{center} for every $F\in\sxi$
and all choices of scalars $(a_i)_{i\in F}$.
\end{defn}
The following is a well known result and for its proof we refer
the interested reader to \cite{AMT}.

\begin{lem}
\label{lxi1} If a separable Banach space $Y$ contains
$\ell_{\xi}^1$-spreading model, for every $\couo$, then $Y$
contains an isomorphic copy of $\ell^1$.
\end{lem}
As we have supposed, the space $X_T$ does not contain $\ell^1$
therefore it follows that there is $\couo$ such that $X_T$ contains
no $\ell_{\xi}^1$-spreading model. We fix this countable ordinal
$\couo$ and we use the following norming set for $X_T$:
\begin{center}$G_1=\{\sum_{t\in F}b_te_t^*:
\|\sum_{t\in F}b_te_t^*\|_{X_T^*}\leq 1\text{ and } F\subseteq T
\text{ finite and segment complete}\}$
\end{center}
We also consider a bijection $h: T\to \nn$ as in Definition \ref{bijection} and make use of the following piece of notation:
\begin{notation} For every sequence $(f_i)_{i=1}^d$ in $c_{00}(T)$ such that
\begin{enumerate} \item[i.]$(f_i)_{i=1}^d$ is block \item[ii.]$\big{\{} \min \{h(t): t\in\spp f_i\}\big{\}}_{i=1}^d\in\sxi$ \item[iii.]
$\{\spp f_i\}_{i=1}^d$ are incomparable subsets of $T$
\end{enumerate} We will call $(f_i)_{i=1}^d$ a
$(T,\xi)$-admissible sequence
\end{notation}
The definition of the norming set is the following
\begin{defn}
\label{dxi2} Let $\gox$ be the minimal subset of
$c_{00}(\tr)$ such that
\begin{enumerate}
\item[1.] $\goo\subseteq\gox$ \item[2.] $\gox$ is closed under the
$(\frac{1}{2}, \mathcal{S}_{\xi})$-operation on
$(T,\xi)$-admissible sequences. That is, for every
$(T,\xi)$-admissible sequence $f_1,...,f_d$ in $\gox$ we have that
$\frac{1}{2}\sum_{i=1}^d f_i$ is an element of $\gox$.
\end{enumerate}
\end{defn}
We define a norm on $c_{00}(T)$ as follows:
\begin{center} For every  $x\in c_{00}(T)$ we let $\|x\|_{\auxx}=\sup\{f(x):
f\in\gox\}$\end{center} and set \begin{center}
$\auxx=\overline{<e_t:t\in T>}^{\|.\|_{\auxx}}$\end{center}
\begin{rem}
\label{P1} It can be readily verified that $(e_t)_{t\in T}$
(enumerated via $h$) becomes a bimonotone
Schauder basis for $\auxx$. In addition as $\goo\subset\gox$ it is
evident that the identity operator $I: \auxx\to X_T$ is continuous.
\end{rem}

We also have the following,
\begin{lem}
\label{lxi2} The set $\gox$ is closed under restrictions of its
elements on segment complete subsets of $T$ and thus for every
segment complete $A\subseteq T$ the natural projection $P_A:
\auxx\mapsto\auxx$ defined by $P_A(\sum_{t\in  T} \lambda_t
e_t)=\sum_{t\in A} \lambda_t e_t$ has norm 1.
\end{lem}
\begin{proof}
Let $f\in\gox$ and $A\subset T$ segment complete. We will show
that $f|A$ by using induction on the complexity of $f$. Suppose
that $f\in\goo$. By our assumptions we have
$f|A\in\goo\subset\gox$. Now let $f=\frac{1}{2}\sum_{i=1}^d
f_i\in\gox$ and assume that $f_i|A\in\gox$ for all $i=1,...,d$.
Then $f|A=\frac{1}{2}\sum_{i=1}^d f_i|A$ and the following
properties of $(f_i|A)_{i=1}^d$ can be readily verified
\begin{enumerate} \item[i.] $\{f_i|A\}_{i=1}^d$ is a block sequence
\item[ii.] $\{\spp f_i|A\}_{i=1}^d$ are pairwise incomparable
subsets of $T$ \item[iii.] $\{\min\{h(t):t\in \spp
f_i|A\}\}_{i=1}^d\in\sxi$, since $\sxi$ is hereditary and
spreading \end{enumerate} Thus $f|A\in\gox$.
\end{proof}
\begin{defn}
\label{dxi3} Let $f\in\gox$. By a tree analysis of $f$ we mean a
finite family $(f_a)_{a\in\mathcal{A}}$ indexed by a finite tree
$\mathcal{A}$ with a unique root $0\in\mathcal{A}$ such that
\begin{enumerate} \item[1.] $f_0=f$ and  $f_a\in\gox$ for every $a\in\mathcal{A}$
\item[2.]An $a\in\mathcal{A}$ is maximal if and only if
$f_a\in\goo$ \item[3.] For every $a\in\mathcal{A}$ not maximal we
denote by $S_a$ the set of immediate successors if $a$ in $A$ and
define an ordering denoted by $<$ on $S_{a}$ with $b_1<b_2$ if and
only if $f_{b_1}<f_{b_2}$ for all $b_1,b_2\in S_{a}$. Then we have
that $(f_b)_{b\in S_{a}}$ ordered by $<$ is a $(T,\xi)$-admissible
sequence and $f_a=\frac{1}{2}\sum_{b\in S_{a}}f_{b}$
\end{enumerate}
\end{defn}
It is straightforward that by the minimality of $\gox$ that every
$f\in\gox$ admits a tree analysis.

 \begin{rem}
 \label{intact} We note that the definition of the norming set
 $\gox$ uses a Tsirelson type extension technique but only on
 functionals with incomparable supports. Therefore, if we consider any branch $b\in [T]$ and a vector
 $x\in\auxx$ such that $\spp x\subset b$ then we can observe that for every
 $f\in\gox$ with a tree analysis $(f_a)_{a\in A}$ there exists at
 most one maximal $a\in A$ such that $\spp f_a\cap \spp
 x\neq\emptyset$. Hence, for every such vector it follows that
 $\|x\|_{X_T}=\|x\|_{\auxx}$. This fact allows us to identify the sets $\kkk\subset X_T$ and $I^{-1}(\kkk)\subset\auxx$. We will use this for what follows.
 \end{rem}

\begin{defn}
\label{dxi4}  Let $(y_n)_{n\in\nn}$ be a block sequence in $\auxx$
and $f\in\gox$. \begin{enumerate} \item[1.]We set $M_f=\{n\in\nn:
\spp f\cap \rg y_n\neq\emptyset\}$ and if
$(f_a)_{a\in\mathcal{A}}$ is a tree analysis of $f$ we define a
correspondence $\lambda_f:M_f\mapsto\mathcal{A}$ with
$\lambda_f(n)$ to be the $\lesseq_{\mathcal{A}}$-maximal element
of $\mathcal{A}$ such that $\spp f_{\lambda_f(n)}\cap \rg y_n=\spp
f\cap \rg y_n$.
 \item[2.] For all $a\in\mathcal{A}$
we define $D_a=\bigcup_{b\in S_{a}}\{n\in\nn: b=\lambda_f(n)\}$,
equivalently, $D_a=\{n\in\nn: \emptyset\neq \spp f_a\cap \rg
y_n=\spp f\cap \rg y_n\}$ and $E_a=D_a\setminus\bigcup_{b\in
S_{a}}D_b$, or equivalently, $E_a=\{n\in\nn:
a=\lambda_f(n)\}$.\item[3.] For $a\in\mathcal{A}$ not
$\lesseq_{\mathcal{A}}$-maximal we set $b_L(n)=\min\{b\in S_{a}:
\spp f_b\cap \rg y_n\neq\emptyset\}$ and $b_R(n)=\max\{b\in S_{a}:
\spp f_b\cap \rg y_n\neq\emptyset\}$ where the maximum and minimum
are taken with respect to the ordering on $S_{a}$ defined above.
\item[4.] For a block sequence $(y_n)_n$ in $\auxx$ we set
$p_n=\min\spp y_n$ and $q_n=\max\spp y_n$, for all $n\in\nn$.
\end{enumerate}
\end{defn}

We start with an easy but crucial observation and will be used extensively in what follows.

\begin{rem}
\label{l1x}
Let $(y_n)_n$ be a seminormalized level block sequence such that $\spp y_n\perp\spp y_m$ for all $n\neq m \in \nn$. Then $(y_n)_n$ is a $\ell^1_{\xi}$ spreading model.
\end{rem}
\begin{proof} Let $r>0$ be such that $\|y_n\|>r$ for all $n\in\nn$. Choose a sequence of functionals $(f_n)_n$ in $\gox$ such that for each $n\in\nn$ the following hold:
\begin{enumerate}
\item[i.] $\rg f_n\subset \rg y_n$;
\item[ii.] $f_n(y_n)>\frac{r}{2}$.
\end{enumerate}
Let $F\in \sxi$ and $(b_i)_{i\in F}\in c_{00}(\nn)$. It is easy to see that the functional $f=\frac{1}{2}\sum_{i\in F}\sgn (b_i)f_i$ belongs to $\gox$. In addition,
\[\|\sum_{i\in F}b_i y_i\|\geq f(\sum_{i\in F}b_i y_i)\geq \frac{r}{2}\sum_{i\in F}|b_i|,\]
which proves that $(y_n)_n$ is a $\ell^1_{\xi}$ spreading model.
\end{proof}
The next proposition is the basic tool for proving that $\kkk$ is
weakly compact in $\auxx$. It is also used in the next
section where we show that the closed convex hull of $\kkk$ is
thin in $\auxx$.
\begin{prop}
\label{bprop} Let $(y_n)_n$ be a bounded level-block sequence in
$\auxx$ such that \[\lim_{n\to\infty}\|y_n\|_{X_T}=0\] Then there
exists subsequence of $(y_n)_n$ which satisfies an upper
$T_{\xi}$-estimate, that is, there exist a constant $C>0$ and
$M\in [\nn]$ such that for every choice of scalars
$(\lambda_i)_{i=1}^k\in$ \con we have
\begin{center} $\|\sum_{i=1}^k \lambda_i y_{m_i}\|_{\auxx}\leq
C\|\sum_{i=1}^k \lambda_i e_{p_{m_i}}\|_{T_{\xi}}$ \end{center}
where $M=\{m_1<m_2<...\}$.
\end{prop}
In order to prove Proposition \ref{bprop} we need the following
Lemma.
\begin{lem}
\label{lxi3} Let $(y_n)_n$ be a bounded level-block sequence in
$\auxx$ with $\|y_n\|_{\auxx}\leq r\text{ }\forall n\in\nn$.
Suppose also that $\sum_{n\in\nn}\|y_n\|_{X_T}<2r$. Then for every
$f\in\gox$ there exists a $g\in B_{T_{\xi}^1}$ such that for every
$k\in\nn$ and every choice of scalars $(\lambda_i)_{i=1}^k\in$
\con we have that $|f(\sum_{i=1}^k \lambda_i y_i)|\leq
2rg(\sum_{i=1}^k \lambda_i e_{q_i})$
\end{lem}
\begin{proof} Let $f\in\gox$, $(f_a)_{a\in\mathcal{A}}$ a tree
analysis of $f$ and $(\lambda_i)_{i=1}^d\in c_{00}(\nn)$. For each
$a\in\mathcal{A}$ with $D_a\neq\emptyset$ we will recursively
define $g_a$ such that the following hold

\begin{enumerate}
 \item[i.] $g_a\in B_{T_{\xi}^1}$
 \item[ii.] $\spp g_a\subseteq \{q_n: n\in D_a\}$
\item[iii.] $|f_a(\sum_{i\in D_a}\lambda_i y_i)|\leq
2rg_a(\sum_{i\in D_a}\lambda_i e_{q_i})$
\end{enumerate} Let
$a\in\mathcal{A}$ be a maximal element of
$\mathcal{A}$ such that $D_a\neq\emptyset$. Let also $n_0\in\nn$
such that $\lambda_{n_0}=\max_{i\in D_a}|\lambda_i|$. We set
$g_a=\sgn\lambda_{n_0}e_{q_{n_0}}^*$. Clearly $g_a\in
B_{T_{\xi}^1*}$ and
 \begin{eqnarray*} |f(\sum_{i\in D_a}
\lambda_i y_i)| \leq  \max_{i\in D_a
}|\lambda_i|\cdot\sum_{i\in D_a}|f(y_i)| &\leq&
|\lambda_{n_0}|\cdot\sum_{n\in\nn}\|y_n\|_{X_T}  \leq
2rg_a(\sum_{i\in D_a}\lambda_i e_{q_i})
\end{eqnarray*}
Let now $a\in\mathcal{A}$ not maximal. Suppose also that for every
$b\in S_{a}$ the functionals $(g_b)_{b\in S_{a}}$ have been
defined satisfying conditions i. ii and iii above. Let
$\{b_1<...<b_l\}$ be the enumeration of $S_{a}$ as it was given in Definition \ref{dxi4}. Pick $b_i\in
S_{a}$ and suppose that $D_{b_i}\neq\emptyset$. Then,
\begin{center} $\min\spp f_{b_i}\leq \spp g_{b_i}<
\min\spp f_{b_{i+1}}$.\end{center} The left inequality holds
because if we pick $k\in D_{b_i}$ then $\max\spp x_k\geq\min\spp
f_{b_i}$ and $\spp g_{b_i}\subseteq \{q_n: n\in D_{b_i}\}$. On the other hand assume that there exists $k\in D_{b_i}$
such that $q_k>\min\spp f_{b_{i+1}}$ then $\rg x_k\cap \spp
f_{b_{i+1}}\neq\emptyset$. This contradicts the definition of
$D_{b_i}$ and proves the right hand inequality. Similarly, we can see that for every $n\in E_a$ such
that $b_R(n)\neq b_l$ we have $\min\spp f_{b_R(n)}\leq q_n<
\min\spp f_{b_R(n)+1}$. For every $i$ with $1\leq i\leq l$ we set,
\begin{center}
$M_i=\spp g_{b_i}\cup\{q_n: n\in E_a \text{ and } b_i=b_R(n)\}$.
\end{center}
We can readily observe the following,

 \begin{enumerate} \item[i.]
For every $b\in S_{a}$ it holds $|\{q_n: n\in E_a \text{ and }
b=b_R(n)\}|\leq 1$;
 \item[ii.] For  $i<l$ we get $\{q_n: n\in
E_a\text{ and } b_i=b_R(n)\}<\spp g_{b_i}$, while for $i=l$ we
have the converse;
 \item[iii.] $\bigcup_{b\in
S_{a}}M_b=\big{(}\bigcup_{b\in S_{a}}\spp g_b\big{)}\cup \{q_n:
n\in E_a\}\subseteq \{q_n: n\in D_a\}$;
 \item[iv.] $\min\spp
f_{b_i}\leq M_{b_i}<\min\spp f_{b_{i+1}}$ for all $i=1,...,l-1$.
\end{enumerate}
Combining these four facts we conclude that the functionals
$(e_{q_n})_{n\in E_a}$ and $(g_{b_i})_{1\leq i\leq l}$ together
form a $[\nn]^{\leq 2}\otimes\sxi$-admissible family. Consequently
the functional $g_a=\frac{1}{2}(\sum_{n\in E_a} e^*_{q_n} +
\sum_{b\in S_{a}}g_b)$ is an element of $B_{T^{1^*}_{\xi}}$. Finally,
\begin{eqnarray*} |f_a(\sum_{i\in D_a} \lambda_i y_i)| & \leq &
|f_a(\sum_{i\in E_a}\lambda_i y_i)| +
|f_a(\sum_{i\in D_a\setminus E_a}\lambda_i y_i)|  \leq\\
\leq r\cdot\sum_{i\in E_a}|\lambda_i|&+&|\frac{1}{2}\sum_{b\in
S_{a}}f_b(\sum_{i\in D_b}\lambda_i y_i)|\leq  r\sum_{i\in
E_a}|\lambda_i|+ r\sum_{b\in S_{a}}g_b(\sum_{i\in D_b}\lambda_i e_{q_i})=\\
&=& 2r g_a(\sum_{i\in D_a}\lambda_i e_{q_i}) \end{eqnarray*}
\end{proof}
We are now ready to prove Proposition \ref{bprop}.\\

\noindent\textit{Proof of Proposition \ref{bprop} }. Since
$(y_n)_n$ is a bounded block sequence in $\auxx$ such that
\[\lim_{n\to\infty}\|y_n\|_{X_T}=0\] we may choose $M\in [\nn]$ and a subsequence
$(y_n)_{n\in M}$ such that \begin{enumerate} \item[i.] $\sum_{n\in
M}\|y_n\|_{X_T}<2r$ \item[ii.] $(y_n)_{n\in M}$ is level-block
\end{enumerate} For  simplicity we denote the subsequence by
$(y_n)_n$ again. Lemma \ref{lxi3} yields,
\[\|\sum_{i=1}^k \lambda_i y_i\|_{\auxx}\leq 2r\|\sum_{i=1}^k
\lambda_i e_{q_i}\|_{T_{\xi}^1}.\]  By Lemma 8 we have,
\[\|\sum_{i=1}^k \lambda_i y_i\|_{\auxx}\leq 2r\|\sum_{i=1}^k
\lambda_i e_{q_i}\|_{T_{\xi}^1}\leq 16r\|\sum_{i=1}^k \lambda_i
e_{q_i}\|_{T_{\xi}}.\] Finally, applying Lemma 9 we obtain,
\[\|\sum_{i=1}^k \lambda_i y_i\|_{\auxx}\leq 16r\|\sum_{i=1}^k \lambda_i
e_{q_i}\|_{T_{\xi}}\leq 48r \|\sum_{i=1}^k \lambda_i
e_{p_i}\|_{T_{\xi}},\] for all choices of $k\in\nn$ and
$(\lambda_i)_{i=1}^k\in$ \con completing the proof. \hfill

 \begin{notation} We set $W_{\xi}^0=co(\kkk\cup - \kkk)$
 and $W_{\xi}=\overline{W_{\xi}^0}^{\|.\|_{\auxx}}$.
 \end{notation}
 \begin{prop}
 \label{wkkk} The set $\kkk$ is weakly compact in $\auxx$
 \end{prop}
 \begin{proof}
 Let $(y_n)_n$ be a sequence in $\kkk$. Clearly we may
 assume that each $y_n$ is finitely supported. First we prove the
 following\\
 \noindent \textbf{Claim} \textit{If $(y_n)_n$ consists of
 incomparably supported vectors then it is weakly null.}\\
 \noindent \textbf{Proof of Claim}\\
 Let $r>0$ be such that $\|y_n\|\geq r>0$ and suppose towards a contradiction that there exist an
 $\epsilon>0$ and a functional $x^* \in (\auxx)^*$ with $\|x^*\|=1$ such that
 $x^*(y_n)> \epsilon$ for all $n$. Now, since $\spp y_n$ are
 incomparable segments of $T$ we may also assume (by passing to a subsequence) that
 $(y_n)_n$ is a level block sequence. Remark \ref{l1x} yields that $(y_n)_n$ is a $\ell^1_{\xi}$ spreading model. As $X_T$ does not contain any $\ell^1_{\xi}$-spreading model, there
exists a sequence $(z_n)_n$ of block convex combinations of
$(y_n)_n$ such that $\|z_n\|_{X_T}\to 0$. By Proposition \ref{bprop}, there exists a subsequence of $(z_n)_n$ (denoted by $(z_n)_n$ again) which satisfies an upper-$T_{\xi}$
estimate. As the space $T_{\xi}$ is reflexive this implies that
$(z_n)_n$ is weakly null. Hence, there are further convex combinations of $(z_n)_n$ that converges norm to zero. This is clearly a contradiction since we have assumed that $x^*(y_n)> \epsilon$
for all $n\in\nn$ and it completes the proof of the claim.\\
Now if $(y_n)_n$ is arbitrary we can assume, by passing to a
subsequence if necessary, that $e^*_t(y_n) \stackrel{n}\to y_t$
for all $t\in T$. Observe that $S=\{t\in\ttk: y_t \neq 0\}$ is a
segment (finite or infinite) of $T$. Thus, we may assume that each
$y_n$ has a decomposition as $y_n= u_n + v_n$ where
\begin{enumerate}
\item[i.] $(u_n)_n$ is a $\subseteq$-increasing sequence of
segments of $S$ \item[ii.] $(v_n)_n$ is a sequence of incomparable
segments of $T$
\end{enumerate}
The previous claim yields that $(v_n)_n$ is weakly null. To finish the
proof set $y= \sum_{t\in S} y_te_t$ and observe that $(u_n)_n$ has a subsequence which converges to $y$ in the norm topology of $\auxx$. Indeed, by the weak compactness of $\kkk$ in $X_T$ there exists a  subsequence (which we denote by $(u_n)_n$ again) such that $u_n\stackrel{(X_T,w)}\to y'\in\kkk$. As
$e^*_t(u_n)\to e^*_t(y)$ we deduce that $y=y'$. It is easy to see that the definition of the set
$\kkk$ implies, in fact, that $u_n\stackrel{\|\cdot\|_{X_T}}\to y$.
Since, $\spp u_n\subset S$ for all $n\in\nn$ we deduce by Remark \ref{intact}
that $u_n\stackrel{\|\cdot\|_{\auxx}}\to y$ which completes the proof.
\end{proof}

\subsection{The set $\kkk$ is thin in $\auxx$.\\}
\label{thinness}

In this subsection use Lemma \ref{lthin1}, Proposition \ref{lthin1} and Proposition \ref{bprop} in order to show that the closed convex hull $W_{\xi}$ of $\kkk$ is thin in $\auxx$. Lemma \ref{lthin1} and Proposition \ref{lthin1} use techniques developed in \cite{AF}, adapted to this setting, which are crucial for the proof. We note that for the sake of simplicity of notation hereby all norms are considered in $\auxx$ unless stated otherwise. We start with the definition of a thin subset of a Banach space.
 \begin{defn}
 \label{dthin1}
   Let $A,\Gamma$ be two subsets of a Banach space $Y$
 \begin{enumerate}
 \item[i.] Let $\epsilon>0$. We say that $\Gamma$
 $\epsilon$-absorbs $A$ if there exists $\lambda>0$ such that
 $A\subseteq \lambda\Gamma + \epsilon B_Y$
 \item[ii.] We say that $\Gamma$ almost absorbs $A$ if $\Gamma$
 $\epsilon$-absorbs $A$ for every $\epsilon>0$
 \item[iii.] We say that a $A$ is thin in $Y$ if $A$ \textit{does
 not} almost absorb the ball of any infinite dimensional closed
 subspace of $Y$
 \end{enumerate}
 \end{defn}

 \begin{defn}
 \label{dthin2}
 Let $A$ be a segment complete subset of $T$ and $\epsilon>0$.
 \begin{enumerate}
 \item[1.] For each $x\in\auxx$ we denote by $Ax$ the natural projection of $x$ onto $A$ and for $s$
 segment of $T$ denote by $x_s=sx_t=\sum_{t\in s} \psi(t)e_t$
 \item[2.] We set $A''=\{t\in A: x_t\in \kkk\}$ and $A^{\epsilon}=\{t\in A'':\|Ax_t\|\geq\epsilon\}$.
 Let also $A'=A\setminus A^{\epsilon}$
 \item[3.] We set $\segm (A)=\{s\text{ segment}: t\in A''\text{ for all } t\in s\}$. Clearly, for all
 $s\in\segm (A)$ it follows that $s\subseteq A$.
 \end{enumerate}
 \end{defn}
 We can readily observe the following:
 \begin{rem}
 \label{E''segment complete}For every $A$ segment complete and $\epsilon>0$ the set $A^{\epsilon}$ is also segment complete.
 \end{rem}
 \begin{proof}
 Let $t_1,t_2\in A^{\epsilon}$ with $t_1\lesseq t_2$. Then we have that
 $\|x_{t_2}\|\leq 1$ thus for every $t\in [t_1,t_2]$ we obtain
 $\|x_t\|\leq \|x_{t_2}\|\leq 1$ and at the same time
 $\|x_{\sigma_t}\|\geq\|x_{\sigma_{t_1}}\|>\epsilon$. These facts
 imply $t\in A^{\epsilon}$.
 \end{proof}

\begin{lem}
 \label{lthin1}
 Let $\epsilon>0$ and $E$ a subset of $T$ of the form
 $E=\{t\in T: m\leq |t|\leq M\}$. Then there exists a
 decomposition of $E$ into two disjoint subsets $E',E''$ such
 that
 \begin{enumerate}
  \item[i.] $\|E'w\|<\epsilon$ for
 every $w\in W_{\xi}^0$ and
 \item[ii.] For $t\in T$ with $|t|\geq M$ we have
 $\|E''x_t\|\geq\epsilon$.
 \end{enumerate}
 \end{lem}
 \begin{proof}
Set $E''=E^{\epsilon}$ as in Definition \ref{dthin2} and
$E'=E\setminus E''$. Observe that for every $w\in W_{\xi}^0$, $Ew$
can be written as $Ew=\sum_{s\in L} \lambda_s x_s$ where
$L\subseteq\segm (E)$ and $\sum_{s\in L} |\lambda_s|\leq 1$. It is
clear that for every $s\in L$ the set $s'=s\cap E'$ is either
empty or a segment of $E$ such that $\|x_{s'}\|<\epsilon$.
Therefore $\|E'w\|=\|\sum_{s\in L}\lambda_s x_{s'}\|<\epsilon$.
 \end{proof}
The following Proposition is the key ingredient for proving that the set
$W_{\xi}$ is a thin subset of $\auxx$. It is an adaptation of the
techniques developed in \cite{AF} and for the sake of completeness
we include its proof here.

 \begin{prop}
 \label{pthin1} Let $(w_n)_n$ be a level-block sequence in
 $W_{\xi}^0$, $\epsilon > 0 $ and $(E_n)_n$ a level-block sequence of subsets of $T$ where
 each one is of the form $E_n=\{t\in T: m_n\leq |t|\leq M_n\}$
 such that $\rg y_n\subseteq E_n$. Then there exist a $L\in [\nn]$
 and a sequence $(F_n)_{n\in L}$ with the following properties:

 \begin{enumerate}
  \item[i.] $F_n\subseteq E_n$ for $n\in L$
 \item[ii.] $(F_n)_{n\in L}$ are pairwise incomparable and segment
 complete subsets of $T$
 \item[iii.] $\|E_nw_n-F_nw_n\|<\epsilon$
 \end{enumerate}
\end{prop}

\begin{proof}
We apply Lemma \ref{lthin1} to find a decomposition of $E_n$ into
two disjoint subsets $E'_n,E''_n$ such that
\begin{enumerate}
\item[i.] $\|E'_nw_n\|_{\auxx}<\frac{\epsilon}{2}$ \item[ii.] If
$s$ is a segment with $s=[t_1,t_2]$ and $|t_1|\leq m_n,|t_2|\geq
M_n$ as well as $s\cap E''_n\neq\emptyset$ we have
$\|E''_nx_s\|\geq\frac{\epsilon}{2}$
\end{enumerate}
Now for every $n\in\nn$ $E_nw_n=\sum_{s\in L_n}\lambda_s x_s$,
where $L_n\subseteq \segm (E_n)$ and $\sum_{s\in L_n}|\lambda_s|\leq
1$. This representation defines a positive measure on $\segm (E_n)$
with $\mu_n(A)=\sum_{s\in A \cap L_n}|\lambda_s|$ for $A\subseteq
\segm (E_n)$. Now let us  consider a probability measure $\nu$ on the
compact metrizable space of the branches $[T]$ of the tree $T$
such that for every segment $s$ and every $\mathcal{O}_s$ basic
clopen neighborhood of $[T]$ that contains $s$ of the form
$\mathcal{O}_s=\{b\in [T]: s\lesseq b\}$ we have that
$\nu(\mathcal{O}_s)>0$. With the help of $\nu$ we define a measure
$\mu$ on $[T]$ as follows. For every clopen $B\subseteq [T]$ we
set
\[\mu(B)=\lim_{n\to\mathcal{U}}\sum_{s\in L_n}
|\lambda_s|\frac{\nu(\mathcal{O}_s\cap B)}{\nu(\mathcal{O}_s)}\]
where the limit is taken with respect to a non-trivial ultrafilter
$\mathcal{U}$ on $\nn$. Using a diagonal argument we may assume
that this is an ordinary limit. Now for every $n<k$ in $\nn$ we
define
\begin{enumerate}
 \item[i.] $B_n=\{b\in [T] : b\cap
E''_n\neq\emptyset\}$ \item[ii.] $A_n^k=\{s\subseteq \segm (E_k):
\hat{s}\cap E''_n\neq\emptyset\}$
 \end{enumerate}
By our definition of $\mu$ we have that
\[\mu(B_n)=\lim_k\mu_k(A_n^k)\]
\noindent\textbf{Claim} For every $M\in [\nn]$ and $\delta>0$ the
set $I_{\delta}=\{n\in M: \mu(B_n)<\delta\}$ is an infinite subset of $M$.\\
\noindent\textit{Proof of claim:} Suppose not. Then there exists
an $M\in [\nn]$ such that for every $n\in M$ it holds
$\mu(B_n)\geq\delta$ and therefore there exists a branch $b\in
[T]$ such that $b\in B_n,\forall n\in M$. This implies $b\cap
E''_n\noe$ for all $n\in M$. Thus the sequence
$(x_{b|n})_{n\in\nn}$ converges norm to
$x_b=\sum_{n\in\nn}\psi({b|n})e_{b|n}$ and
$\|E''_nx_b\|\geq\frac{\epsilon}{2}$ for all $n\in M$ which is a
contradiction and proves the claim. \hfill\\
 Now we define
the following sets. First we set $I_0=\nn$ and inductively for
$k\geq
0$ \begin{eqnarray*} \tilde{I}_{k+1} & = & \{n\in I_k:\mu(B_n)<\frac{\epsilon}{C\cdot 2^{k+2}}\}, n_{k+1}=\min\tilde{I}_{k+1} \\
I_{k+1} & = & \{l\in\tilde{I}_{k+1}: \mu_{n_{k+1}}(A_{n_{k+1}}^l)<\frac{\epsilon}{C\cdot 2^{k+2}}\}\\
F_{n_{k+1}}& = & \{t\in E''_{n_{k+1}}: \hat{t}\cap(\cup_{i=1}^k
E''_{n_i})=\emptyset\}=\{t\in E''_{n_{k+1}}:t\perp E''_{i}\text{
for }i\leq n_k\}
\end{eqnarray*} Where by $s_t$ we denote the unique initial
segment that contains $t$. By the previous claim the sets
$\tilde{I}_{k+1},I_{k+1}$ are infinite and since $n_{k+1}\in I_k$
we have that
$\mu_{n_{k+1}}(A_{n_k}^{n_{k+1}})<\frac{\epsilon}{2^{k+1}}$. The
set $I_{\infty}=\{n_1,n_2,...\}$ is infinite and we observe that
$\{F_n\}_{n\in I_{\infty}}$ are incomparable by definition and are
also segment complete. Recall also that for every $x\in\kkk$ it holds $\|x\|\leq C$.
Now let $k\in\nn$ then it remains to show that
$\|(E_{n_k}\setminus F_{n_k})w_{n_k}\|<\epsilon$. For $k>1$ we set
$r=n_{k+1}$. We consider the set $A_r=A_{n_1}^r\cup...\cup
A_{n_k}^r$. Then
$\mu_r(A_r)\leq\sum_{i=2}^{k+2}\frac{\epsilon}{C\cdot
2^{k+2}}<\frac{\epsilon}{4C}$. Let $s\notin A_r$ then $\forall
t\in s$ and $i=1,...,k$ we have $s_t\cap E''_{n_i}=\emptyset$ and
thus $s\subseteq F_r$ and if $s\in A_r$ then $s\cap
F_r=\emptyset$. So
\[\|(E''_r\setminus F_r) w_r\|\leq\sum_{s\in L_r\cap
A_r}|\lambda_s|\cdot \|(E''_r\setminus
F_r)x_s\|\leq\mu_r(A_r)<\frac{\epsilon}{2}\] Thus
\[\|(E_r\setminus F_r)w_r\|\leq\|E'_rw_r\|+\|(E''_r\setminus
F_r)w_r\|<\frac{\epsilon}{2}+\frac{\epsilon}{2}=\epsilon\]
\end{proof}
\begin{prop}
\label{pthin2} Let $(z_n)_n$ be a normalized level-block sequence
in $\auxx$ such that the unit ball $B_Z$ of the subspace
$Z=\overline{<z_n: n\in\nn>}^{\|.\|}$ is almost absorbed by
$W_{\xi}$. Then every normalized block sequence in $Z$ has a subsequence which is a $\ell^1_{\xi}$ spreading model. Moreover, the identity operator $I:Z\mapsto X_T$ is strictly
singular.
\end{prop}
\begin{proof}
Let $(y_n)_n$ be a normalized level-block sequence in $Z$.
By our hypothesis there exists a $\lambda>0$ such that
$B_Z\subseteq\lambda W_{\xi}+\frac{1}{8}B_{\auxx}$. Thus for every
$k\in\nn$ there exists $w_k\in W_{\xi}^0$ such that $\|y_k-\lambda
w_k\|<\frac{1}{8}$. By Proposition \ref{pthin1} there
exists a $M\in[\nn]$ and a sequence $(E_k)_{k\in M}$ of subsets of
$T$ such that the following hold:
\begin{enumerate} \item[i.] $\|w_k-E_kw_k\|<\frac{1}{8\lambda}$
\item[ii.] $(E_k)_k$ are pairwise incomparable and segment
complete subsets of $T$, for all $k\in M$
\end{enumerate}
For the sake of simplicity of notation we assume that $M=\nn$. Now
if we set $w'_k=E_kw_k$ and $w''_k=\lambda w'_k$, we have that
$\|y_k-w''_k\|<\frac{1}{8}+\frac{1}{8}=\frac{1}{4}$. Thus
$\|w''_k\|_{\auxx}\geq\frac{3}{4}$. Remark \ref{l1x} yields that $(w''_k)_k$ is a $\ell^1_{\xi}$ spreading model. It is easy to see that this property is transferred to $(y_n)_n$ as well. In addition, as $X_T$ does not contain $\ell^1_{\xi}$ spreading models, it is immediate that $I:Z\mapsto X_T$ is strictly singular.
\end{proof}

An immediate consequence of the preceding Proposition is the following.
\begin{cor}
\label{cthin1}
Let $(z_n)_n$ be a normalized level-block sequence
in $\auxx$ such that the unit ball $B_Z$ of the subspace
$Z=\overline{<z_n: n\in\nn>}^{\|.\|}$ is almost absorbed by
$W_{\xi}$. Then every normalized block sequence in $Z$ has a further block subsequence which satisfies an upper $(T,\xi)$ estimate.
\end{cor}
\begin{proof}
Proposition \ref{pthin2} yields that for every normalized block sequence $(y_n)_n$ in $Z$ there exists a further normalized block subsequence $(x_n)_n$ of $(y_n)_n$ such that $\|x_n\|_{X_T}\to 0$. A direct application of Proposition \ref{bprop} yields the result.
\end{proof}

We are now ready to prove the main result of this section.

\begin{thm}
\label{Tthin} The set $W_{\xi}$ is thin in $\auxx$.
\end{thm}
\begin{proof} Suppose not. Then there exists a normalized block
sequence $(y_n)_n$ in $\auxx$ such that $B_Y$ is almost absorbed
by $W_{\xi}$, where by $Y$ we denote
$Y=\overline{<y_n:n\in\nn>}^{\|.\|}$. By Corollary \ref{cthin1} we can find a normalized block sequence $(z_n)_n$ in $Y$ such that
 \[\|\sum_{i=1}^k b_i z_{n}\|\leq\
48\|\sum_{i=1}^k b_i e_{p_{n}}\|_{T_{\xi}}\]
for every choice of scalars $(b_i)_{i=1}^k$ and $k\in\nn$. Since by our hypothesis the
unit ball of $Z=\overline{<z_n:n\in\nn>}^{\|.\|}$ is
almost absorbed by $W_{\xi}$ we can apply the same arguments as in
Proposition \ref{pthin2} to obtain a sequence $(z_{m_i}^*)_i$
satisfying the following \begin{enumerate} \item[i.] $\spp
z_{m_i}^*\subseteq \rg z_{m_i}$ for all $i\in\nn$ \item[ii.] $(\spp
z_{m_i}^*)_n$ are pairwise incomparable, segment complete subsets of
$T$ \item[iii.] $z_{m_i}^*(z_{m_i})>\frac{1}{4}$ for all $i\in\nn$
\item[iv.] $z^*_{m_i}\in\gox$ for all $i\in\nn$
\end{enumerate}
Define an operator $P:\auxx\mapsto\overline{<z_n:n\in\nn>}^{\|.\|}$ by $P(x)=\sum_{n=1}^{\infty}z_n^*(x)z_n$. We will show first that $P$ is bounded.

To see this, let $x\in B_{\auxx}$. It is enough to prove that $\|\sum_nz_n^*(x)e_{p_n}\|_{T_{\xi}}\leq 1$. Indeed, Let $f$ be a functional in the norming set of $T_{\xi}$ and $(f_a)_{a\in A}$ a tree analysis of $f$. We can assume without loss of generality that
$\spp f\subseteq\{p_n:n\in\nn\}$. Let $a\in\mathcal{A}$ be a
$\lesseq_{\mathcal{A}}$-maximal node. Then $f_a=\pm e_{p_k}$ for a
$k\in\nn$. Thus $|f_a(\sum_n z_n^*(x)e_{p_n})|=|z_k^*(x)|\leq 1$.
We move on to recursively define for each $a\in\mathcal{A}$ a
functional $g_a\in\gox$ such that $f_a(\sum_n
z_n^*(x)e_{p_n})=g_a(x)$. If $a$ is maximal and $f_a=\pm
e_{p_k}^*$ we set $g_a=z_k^*$ or $g_a=-z_k^*$ respectively. We
observe that if $\{\pm e_{p_{i_1}}^*<...<\pm e_{p_{i_l}}^*\}$ is
an $\sxi$-admissible sequence of functional in $B_{T_{\xi}}^*$
then $\{\pm z_{p_{i_1}}^*<...<\pm z_{p_{i_l}}^*\}$ is an
$(T,\xi)$-admissible family of functionals as well. Now suppose
that $a\in\mathcal{A}$ is not maximal such that for every $\beta\in
S_{a}$ we have defined a functional $g_{\beta}\in\gox$ such that
$(g_{\beta})_{\beta\in S_{a}}$ are successive and their supports
are pairwise incomparable subsets of $T$ and each $g_{\beta}$
satisfies $f_{\beta}(\sum_n z_n^*(x)e_{p_n})=g_{\beta}(x)$ for all
$\beta\in S_{a}$. Then the functional
$g_a=\frac{1}{2}\sum_{\beta\in S_{a}}g_{\beta}$ is an element of
$\gox$ and \[f_a(\sum_n z_n^*(x)e_{p_n})=\frac{1}{2}\sum_{\beta\in
S_{a}}f_{\beta}(\sum_n z_n^*(x)e_{p_n})=\frac{1}{2}\sum
g_{\beta}(x)=g_a(x).\] Thus, by following the structure of the tree
$\mathcal{A}$ we arrive at a functional $g\in\gox$ with
$g(x)=f(\sum_n z_n^*(x)e_{p_n})$ and this gives us
\[\|\sum_n z_n^*(x)e_{p_n}\|_{T_{\xi}}\leq \|x\|\leq 1\] Therefore, $\|P\|\leq
C$.

Suppose now, that $\forall\epsilon>0$, $\exists\lambda>0$ such
that
\[B_Y\subseteq\lambda W_{\xi}+ \epsilon B_{\auxx}\] Then it is
clear that $B_Z\subseteq\lambda W_{\xi}+ \epsilon B_{\auxx}$ where
$Z=\overline{<z_n:n\in\nn>}^{\|.\|}$. Since
$|z_n^*(z_n)|>\frac{1}{4}$ we obtain that
\[\frac{1}{4}B_Z\subseteq P(B_Z)\] and by setting
$\epsilon=\frac{1}{8\|P\|}$ we have \[P(B_Z)\subseteq\lambda
P(W_{\xi})+ \frac{1}{8}B_Z\] Thus $B_Z\subseteq
8\lambda\overline{P(W_{\xi})}^{\|.\|}$. Since the operator
$P$ is defined by the sequence $(z_n^*)_n$ and $(\spp z_n^*)_n$
are pairwise incomparable we have that
$P(W_{\xi})\subseteq\|P\|\overline{co}[(\pm z_n): n\in\nn]$ and
finally \[B_Z\subseteq 8\lambda\|P\|\overline{co}[(\pm
z_n):n\in\nn]\]
 Now since, as is well known, the basis of
$T_{\xi}$ is weakly null we can select a convex combination
$x=\sum_{i=1}^n k_ie_{p_i}$ such that
$\|x\|_{T_{\xi}}<\frac{C}{16\lambda\|P\|}$ and we obtain
$\|\sum_{i=1}^n k_i z_i\|_{\auxx}<\frac{1}{16\lambda\|P\|}$. We
observe that if $z=\sum_{i=1}^n k_i z_i$ then
$(16\lambda\|P\|)z\in B_Z\subseteq 8\lambda\|P\|\overline{co}[(\pm
z_n):n\in\nn]$. We conclude
that $W_{\xi}$ \textit{does not} almost absorb $B_Y$. This is a contradiction which yields the proof of the Theorem.
\end{proof}

\section{Classical Interpolation Spaces}
\label{int1}
 We fix $\tr$,  $X_{\tr}$ with a $\scu$-unconditional basis $(e_t)_{t\in\tr}$ and $\kkk$ as in the previous section.
 We also set $W=\overline{\mathrm{co}}^{\|\cdot\|_{X_{\tr}}}(\kkk\cup-\kkk)$. In this section we use the classical Davies-Figiel-Johnson-Pelczynski iterpolation method \cite{DFJP} for the pair $(X_{\tr},W)$ to produce a new space $\auxb$ in which the structure of the set $\kkk$ is preserved and study the properties of this new space.
 We begin by recalling the (DFJP)-interpolation method:

 \begin{defn}
 \label{dint11} Let $\tr$, $X_{\tr}$ and $W$ be as above. We set $W_n=2^n W+\fr B_{X_{\tr}}$ and define a sequence of equivalent
 norms $(\|\cdot\|_n)_n$ on $X_{\tr}$, each induced by the Minkowski gauge of the respective $W_n$. We consider $(\sum_{n=1}^{\infty}\oplus(X_{\tr},\|\cdot\|_n))_2$ to be the
$\ell_2$-Schauder sum of the spaces $(X_{\tr},\|\cdot\|_n)_n$. Finally, we
set $\auxb$ to be the diagonal space of this $\ell_2$-Schauder
sum. That is, the (closed) subspace consisting of all elements of the form $\tilde{x}=(x,x,...x,...)$, for $x\in X_{\tr}$.
We also denote by $J_1$ the $1-1$ bounded linear operator $J_1:\auxb\to X_{\tr}$ defined as $J_1(\tilde{x})=x$ and by
$\tilde{\kkk}=\{\tilde{x}:x\in\kkk\}$.
\end{defn}
\begin{rem}
\label{rint11} In \cite{DFJP} it was proved that if the set $W$ is weakly compact then the space $\auxb$ is reflexive. In addition, by its construction the space $\auxb$ consists of all elements $\tilde{x}=(x,x,...,x,...)$ such that
$\sum_{n=1}^{\infty}\|x\|_n^2<\infty$ and
$\|\tilde{x}\|_{\auxb}=(\sum_{n=1}^{\infty}\|x\|_n^2)^\frac{1}{2}$. Therefore,
we can observe that for $w\in W$ we have $\|w\|_n\leq \fr$ and
thus $\|\tilde{w}\|_{\auxb}\leq 1$. It follows that $\tilde{\kkk}$ is a closed subset of $\auxb$ and $J_1(\tilde{\kkk})=\kkk$.
\end{rem}
We pass now to show that the space $\auxb$ has a $\scu$-unconditional basis. We start with the following Lemma.
\begin{lem}
\label{lint11}
Let $\tr_{\kkk}=\{t\in\tr:y_s=\sum_{t'\lesseq t}\psi(t')e_{t'}\in\kkk\}$. Then $\tr_{\kkk}$ is a backwards closed subtreee
of $\tr$ and for every $t\in\tr_{\kkk}$ we have that
$\tilde{e}_t=(e_t,e_t,...)\in\auxb$.
\end{lem}

\begin{proof}
It is easy to check that $\tr_{\kkk}$ is indeed a backwards closed subtree of $\tr$. Let now $s$ be an initial segment of $\tr$ (finite or infinite) such that the vector $y_s=\sum_{t\in s}\psi(t) e_t$ is an
element of $\kkk$. It follows by the definition of $\kkk$ that
$\psi(t)e_t\in\kkk$, $\forall t\in s$. Let $t_0\in\tr_{\kkk}$, then $\sum_{t\lesseq t_0}\psi(t)e_t\in\kkk$ and we obtain  $\psi (t)e_t\in\kkk$ $\forall t\lesseq
t_0$. So $a_{t_0}e_{t_0}\in\kkk$ and thus
$\tilde{e}_{t_0}\in\auxb$
\end{proof}

\begin{rem}
\label{segcint1}
We note that as $\tr_{\kkk}$ is a backwards closed subtree of $\tr$ then for every $A\subset\tr_{\kkk}$ segment complete we have
that $A$ is also segment complete when considered as a subset of $\tr$.
\end{rem}

We fix a bijection $g: \tr_{\kkk}\mapsto\nn$  as in Definition \ref{bijection}
and we pass to show that the sequence $(\tilde{e}_t)_{t\in\tr_{\kkk}}$
enumerated through $g$ defines a bimonotone Schauder basis for
$\auxb$.

\begin{lem}
\label{lint13} Let $A\subseteq\tr$ segment complete. If we denote
by $P_A:{X_{\tr}}\mapsto \overline{<e_t:t\in A>}^{\|.\|_{X_{\tr}}}$
the natural projection induced by $A$ we have that
$\|P_A(x)\|_n\leq \|x\|_n$ for all $x\in X_{\tr}$ and $n\in\nn$.
\end{lem}
\begin{proof}
As $X_{\tr}$ has a $\scu$-unconditional basis it follows that $P(B_{\auxa})\subseteq B_{\auxa}$. At the same time for every
$w\in W$  we have $\|P_A(w)\|_{X_{\tr}}\leq \|w\|_{X_{\tr}}$ and
$P_A(w)\in W$. Thus, $P_A(W)\subseteq W$. Let now $n\in\nn$ and
$x\in X_{\tr}$. Let also $\lambda>0$ such that $x\in \lambda(2^nW+\fr
B_{X_{\tr}})$. All the above yield $P_A(x)\in \lambda(2^nP_a(W)+\fr
P_A(B_{X_{\tr}})\subseteq \lambda(2^nW+\fr B_{X_{\tr}})$. Thus
$\|P_A(x)\|_n \leq \|x\|_n$ as desired.
\end{proof}

\begin{prop}
\label{basisint1}The sequence $(e_t)_{t\in\tr_{\kkk}}$ is a $\scu$-unconditional Schauder basis for $\auxb$.
\end{prop}
\begin{proof}
By Remark \ref{segcint1} and the previous Lemma it readily follows that $(e_t)_{t\in\tr_{\kkk}}$ is a $\scu$-unconditional Schauder basis for the subspace
$E=\overline{<(\tilde{e}_t)_{t\in\tr_{\kkk}}>}$. We shall show that $E$ actually coincides
with $\auxb$. We need the following Claim:\\

\noindent \textbf{Claim} \textit{ For every element
$\tilde{x}=(\sum_{t\in\tr_{\kkk}} \lambda_t e_t, \sum_{t\in\tr_{\kkk}}
\lambda_t e_t,...)\in\auxb$, we have that $\tilde{x}\in E$}.\\
\noindent \textbf{Proof of claim} Let
$\tilde{x}=(\sum_{t\in\tr_{\kkk}} \lambda_t e_t,
\sum_{t\in\tr_{\kkk}} \lambda_t e_t,...)$ then
$\|\tilde{x}\|_{\auxb}= (\sum_n \|\sum_{t\in\tr_{\kkk}} \lambda_t
e_t \|^2_n)^{\frac{1}{2}}$. Let $\epsilon>0$. There exists a $n_0
\in\nn$ such that $(\sum_{n> n_0} \|\sum_{t\in\tr_{\kkk}}
\lambda_t e_t\|^2_n)^{\frac{1}{2}}\leq \frac{\epsilon}{2}$. By
Lemma \ref{lint13} and the fact that the spaces
$(\auxb,\|\cdot\|_n)_n$ are mutually isomorphic we can choose a
finite interval $I$ of $\tr_{\kkk}$ such that
\[(\sum_{i=1}^{n_0}\|\sum_{t\in\tr_{\kkk}} \lambda_t e_t -
\sum_{t\in I} \lambda_t e_t\|^2_n)^{\frac{1}{2}}
\leq\frac{\epsilon}{2}\] Thus, if we set $\tilde{x}'=
\sum_{t\in\I} \lambda_t \tilde{e}_t \in E$
we obtain $\|\tilde{x}- \tilde{x}'\|_{\auxb} \leq \epsilon$.\\
This
completes the proof of the claim.\\

Let now, $\tilde{x}\in \auxb\setminus E$. We know that $\tilde{x}$
is of the form $\tilde{x}=(\sum_{t\in\ttt} \lambda_t e_t,
\sum_{t\in\ttt} \lambda_t e_t,...)$. Since $x \notin E$ we have
that there exists $t_0 \in\ttt\setminus\tr_{\kkk}$ such that
$\lambda_{t_0}\neq 0$. Let $w \in W$ and $n\in\nn$. We have that
$w$ is of the form $w=\sum_{t\in \tr_{\kkk}} \beta_t e_t$. Hence,

\begin{equation*}
\|\sum_{t\in\ttt} \lambda_t e_t - 2^n w\|_{\auxa}=
\|\sum_{t\in\ttt} \lambda_t e_t - 2^n \sum_{t\in \tr_{\kkk}} \beta_t
e_t\|_{\auxa} \geq |\lambda_{t_0}|> 0
\end{equation*}
Thus, $\|\sum_{t\in \ttt} \lambda_t e_t\|_n \nrightarrow 0$ and
consequently $\tilde{x}\notin \auxb$. This is a contradiction completing the proof.
\end{proof}

The following result is included in \cite{AF} and for the sake of
completeness we outline the main arguments of its proof.
\begin{prop}
\label{j1} If $W$ is a thin subset of $X_{\tr}$ then the operator
$J_1$ is strictly singular and every infinite dimensional closed
subspace $Y$ of $\auxb$ contains an isomorphic copy of $\ell^2$
which is complemented in $\auxb$.
\end{prop}
\begin{proof}
In order to show that the operator $J_1$ is strictly singular we shall in fact prove something stronger,
namely that $J_1(B_{\auxb})$ is a thin subset of $X_{\tr}$. This is a direct consequence of the following: \\
The set $J_1(B_{\auxb})$ is almost absorbed by $W$.
To see this let $\epsilon>0$.
 Fix $n_0\in\nn$ so that $\frac{1}{2^{n_0}}$ and pick an arbitrary $\tilde{x}=(x,x,...)\in B_{\auxb}$.
 Then, $\sum_{n\in\nn}\|x\|_n^2\leq 1$ which implies that $x\in 2^n W+\frac{1}{2^n}B_{X_{\tr}}$ for all $n\in\nn$.
 Simply set $\lambda=2^{n_0}$ and observe that $J_1(\tilde{x})=x\in 2^{n_0}W+\frac{1}{2^{n_0}} B_{X_{\tr}}\subset \lambda W+\epsilon B_{X_{\tr}}$.
 Now as $J_1(B_{\auxb})$ is almost absorbed by a thin subset it is straightforward that this set is also thin in $X_{\tr}$.
 Pick an arbitrary $Y$ closed subspace of $\auxb$.
 Since the operator $J_1$ is strictly singular one can apply a standard sliding hump argument to produce normalized sequences $(\tilde{y}_n)_n$
 in $Y$  and $(\tilde{z}_n)_n$ horizontally block in
 $\big(\sum_{n\in\nn}\oplus(X_{\tr},\|\cdot\|_n)\big)_2$ such
 that $\sum_{n=1}^{\infty} \|\tilde{z}_n-\tilde{y}_n\|<\frac{1}{2}$. As $(\tilde{z}_n)_n$ is
 isometric to the standard $\ell_2$ basis the space
 $Z=\overline{<\tilde{z}_n:n\in\nn>}$ is 1-complemented in $\auxb$, we
 conclude that the space generated by $(\tilde{y}_n)_n$ is isomorphic to
 $\ell_2$ and complemented in $\auxb$.
\end{proof}

\begin{rem} \label{dfjprem} We note that under the obvious
modifications the results presented in this section remain valid
for DFJP $\ell_p$ interpolation.
\end{rem}

\section{Reflexive spaces as quotients of $\ell^p$ saturated spaces.}
\label{quotients}
At this point we are able to use the techniques developed in all the previous sections in order to show that every separable reflexive Banach space $X$ is a quotient of a separable reflexive and $\ell^p$-saturated space,
for every $p\geq 1$ and of a separable $c_0$-saturated space.
This is done by using all of the results obtained above in conjunction with the following well known result of Zippin (\cite{Z}).

\begin{thm}
\label{zippin} Let $X$ be a separable reflexive Banach space. Then there exists a reflexive Banach space $Z_X$ with a Schauder basis $(z_i)_i$ so that $X$ is isomorphic to a subspace of $Z$.
\end{thm}
We pass now to show the main result of this section. We present the arguments only in the case of $p=2$ as for any $p\geq 1$ and $c_0$ the proof follows exactly the same lines. Namely, we have
\begin{thm}
\label{quotientofell2} Let $X$ be a separable reflexive Banach
space. Then for every $p\geq 1$ there exists a separable reflexive
complementably $\ell^p$-saturated Banach space $X_p$ so that $X$
is isomorphic to a quotient space of $X_p$. Also there exists a
separable $c_0$-saturated space $X_0$ so that $X$ is a quotient of
$X_0$.
\end{thm}
\begin{proof} Granting Zippin's theorem above we may assume that $X$ has a normalized and bimonotone Schauder basis $(x_i)_i$. Starting with $X$ we consider the space $\auxa$ associated to $X$ as it was presented in section \ref{tree}. We also consider the set $\kkk$ and the map $\Phi:\auxa\to X$ (see Definitions \ref{thesetkkk} and \ref{themapf}). By Proposition \ref{dense} we know that $\Phi(\kkk)$ is a $\frac{1}{8}$-net in the unit ball of $X$ and hence $\Phi$ is onto. By Proposition \ref{prp2} the set $W^0=\overline{\co}^{\|\cdot\|}(\kkk\cup -\kkk)$ is a weakly compact subset of $\auxa$. Therefore, the space $\auxb$ as it was defined in the previous section is a reflexive Banach space with a basis $(\tilde{e}_t)_{t\in \mathcal{T}_{\kkk}}$ (see Proposition \ref{basisint1}. In addition, by Remark \ref{rint11} we have $J_1(\tilde{\kkk})=\kkk$. Hence, the operator $\Phi\circ J_1:\auxb\to X$ is onto. By using the extension technique of section \ref{thin} on the space $\auxb$ we arrive at a space $\auxx$ with the properties that $I:\auxx\to \auxb$ is continuous, $\tilde{\kkk}\subset \auxx$ and $W^{\xi}=\overline{\co}^{\|\cdot\|}(\tilde{\kkk}\cup -\tilde{\kkk})$ is weakly compact and thin. Finally, by applying the DFJP - interpolation to the the space $\auxx$ with the set $W^{\xi}$ we arrive at the space $X_2$. The map
$J_2:X_2\to \auxx$ is continuous and preserves the set
$\tilde{\kkk}$. Therefore, there is a map $\Pi:X_2\to X$ onto. To
complete the proof, we point out that by Propositions
\ref{basisint1} and \ref{j1} the space $X_2$ is separable
reflexive and complementably $\ell^2$-saturated.
\end{proof}

\begin{rem}\label{HIinellp} The above Theorem yields examples of
pairs $(X,X^*)$ of reflexive spaces with divergent structure.
Namely, there exists spaces $X_p$ as above such that the dual
$X^*_p$ contains HI subspaces.
\end{rem}
\section{ Skew HI interpolation }
In this section we present a method for applying HI interpolation to a pair $(X,W)$ in order to achieve the diagonal space to have a Schauder basis. We start with a tree $T$, a reflexive space $X_T$ which has a $\scu$-unconditional basis $(e_t)_{t\in T}$ and a weakly compact convex symmetric subset $W$. We denote by $(X_n)_n$ the sequence of (mutually isomorphic) spaces $(X_T,\|\cdot\|_n)$ where the n-th norm is defined via the Minkowski gauge of the set $2^nW+\frac{1}{2^n}B_{X_T}$  and prove the following:

 \begin{thm} \label{skewthm} Let $(X_n)_n$ be the above sequence. Then there exists a norm $\|\cdot\|_G$ defined on $c_{00}(T\times\nn)$ such that if we denote by $\xfr_G$ the completion of $c_{00}(T\times\nn)$ under this norm the following hold:
 \begin{enumerate}
 \item[1.] The sequence $(X_n)_n$ is a Schauder decomposition of $\xfr_G$.
 \item[2.] Setting $Z_t=\overline{<(e_{(t,k)}:k\in\nn>}^{\|\cdot\|_G}$, the sequence $(Z_t)_{t\in T}$ also defines a Schauder decomposition of $\xfr_G$.
 \item[3.] If for all segment complete $A\subset T$ for the natural projection $P_A$ we have that $P_A(W)\subset W$ then the diagonal subspace $\xfr$ of $\xfr_G$ consisting of all elements of the form $\bar{x}=(x,x,...)$ has a Schauder basis. Moreover, $\xfr$ is reflexive.
     \item[4.] If the set $W$ is thin in $X_T$ then the space $\xfr$ is HI.
     \end{enumerate}
     \end{thm}

Let $(X_T,W)$ be as above and assume also that for every segment complete subset $A$ of $X_T$, $P_A(W)\subset W$. Then by Lemma \ref{lint13} we have that for every $n\in\nn$, the sequence $(e_t)_{t\in T}$ is a $\scu$-unconditional Schauder basis for $X_n$. Thus setting:

\[G_{n}=\{\sum_{t\in A}\lambda_t e_t^*: \lambda_t\in\qq \text{
and } \|\sum_{t\in T}\lambda_t e_t^*\|_{X^*_n}\leq 1 \text{
and } A\subseteq T \text{ segment complete}\},\]
we can readily verify that $G_n$ is a norming set for the space $(X_n)$  for all $\nn$.

\begin{notation} We define the following \begin{enumerate}
\item[i.] We set $\pi:T\times\nn\mapsto T$ by $\pi((t,k))=t$
and $j:T\times\nn\mapsto\nn$ by $j((t,k))=k$ \item[ii.] For
$x\in c_{oo}(T\times\nn)$ we let $\rg x$ denote the minimal
rectangle $I\times J$ that contains the support of $x$. Where by a
rectangle $I\times J$ we mean the product of an interval $I$ of
$T$ and an interval of $\nn$. \item[iii.] Let
$A,B\subseteq T\times\nn$ we write  $A\lesss_{\pi}B$ if
$\pi(A)<\pi(B)$ and $A\lesss_{j}B$ if $j(A)<j(B)$ and
$A\lesss_{(\pi,j)}B$ if $\pi(A)<\pi(B)$ and $j(A)<j(B)$. With
$A\lesss^l_{\pi}B$ we denote the property $\pi(A)<^l\pi(B)$
\item[iv.] For $x,y\in c_{00}(T\times\nn)$ we write
$x\lesss_{\pi}y$ whenever $\spp x\lesss_{\pi} \spp y$. The
notations $x\lesss_j y$ and $x\lesss_{(\pi,j)} y$ have analogous
meanings.
\end{enumerate}\end{notation}
\begin{defn}
\label{dhi2} A sequence $(x_n)_n$ in $c_{00}(T\times\nn)$ is
said to be $j$-block ($\pi$-block or level-$\pi$-block) if
$x_n\lesss_j x_{n+1}$ ($x_n\lesss_{\pi} x_{n+1}$ or
$x_n\lesss^l_{\pi} x_{n+1}$ respectively). The sequence $(x_n)_n$
is called diagonally block if $x_n\lesss_{(\pi,j)} x_{n+1}$
\end{defn}
We fix two sequences of natural numbers $(m_l)_{l\in\nn}$ and
$(n_l)_{l\in\nn}$ which are both recursively defined as follows.
We set $m_1=2,m_{l+1}=m_l^5$ and $n_1=4,n_{l+1}=(5n_l)^{s_l}$
where $s_l=\log_2 m_{l+1}$.
\begin{defn}
\label{thesetg} We consider a subset $G$ of
$c_{00}(T\times\nn)$ that is the minimal set such that the
following hold.
\begin{enumerate}
\item[i.] $\bigcup_n \gxn\subseteq G$ and $G$ is closed in the
restriction on rectangles of the form $I\times J$ where $I,J$ are
intervals of $T$ and  $\nn$ respectively.(i.e. for $f\in G$ and
$I,J$ intervals of, we have that $(I\times J)\cdot f=\chi_{I\times
J} \cdot f\in G$). \item[ii.] For every $l\in\nn$, $G$ is closed
in the $\big(\aaa_{n_{2l}},\frac{1}{m_{2l}}\big)$-operations on
$j$-block sequences. That is, if $f_1\prec_j f_2 \prec_j ...
\prec_j f_{n_{2l}}$, then $\frac{1}{m_{2l}} \sum_{i=1}^{n_{2l}}
f_i\in G$. \item[iii.] For every $l\in\nn$, $G$ is closed in the
$\big(\aaa_{n_{2l-1}},\frac{1}{m_{2l-1}}\big)$-operation on
$(n_{2l-1})$-special sequences. \item[iv.] $G$ is rationally
convex.
\end{enumerate}
\end{defn}
It remains to define the $(n_{2l-1})$-special sequences, defined
through a coding $\sg$. For every $l\in\nn$ if $f\in G$ is the
result of the $\big(\aaa_{n_l},\frac{1}{m_l}\big)$-operation, then
we let the weight $w(f)$ of $f$ to be $m_l$. Notice that $w(f)$ is
not uniquely defined.\\

\bigskip

\noindent \textbf{The coding function $\sg$.} First we consider
the subset of $c_{00}(T\times\nn)$ defined by
\begin{eqnarray*}
\mathcal{S}=\big\{ (\phi_1,\phi_2,...,\phi_d) & : & \phi_1\prec_j
\phi_2\prec_j ... \prec_j \phi_d \text{ and }
\phi_i(t,k)\in\mathbb{Q} \text{ for every }\\
& & (t,k)\in T\times \nn \text{ and every } i\in\{1,...,d\}
\big\}.
\end{eqnarray*}
We fix a pair $\Omega_1, \Omega_2$ of disjoint infinite subsets of
$\nn$. As $\mathcal{S}$ is countable, we are able to define an
injection $\sg:\mathcal{S}\to \{2l:l\in\Omega_2\}$ such that
\[ m_{\sg(\phi_1,...,\phi_d)}>\max\Big\{\frac{1}{|\phi_i(e_{(t,k)})|}:
(t,k) \in\spp \phi_i \text{ and } i=1,...,d \Big\}\cdot \max\{k:
(t,k)\in\spp\phi_d\}\]
A finite sequence
$(f_i)_{i=1}^{n_{2l-1}}$ is said to be a $(n_{2l-1})$-special
sequence, provided that
\begin{enumerate}
\item[(a)] $(f_1,...,f_{n_{2l-1}})\in\mathcal{S}$ and $f_i\in G$
for every $i=1,...,n_{2l-1}$, \item[(b)] $w(f_1)=m_{2k}$ with
$k\in \Omega_1$, $m_{2k}^{1/2}>n_{2l-1}$ and
$w(f_{i+1})=m_{\sg(f_1,...,f_i)}$ for every $1\leq i<n_{2l-1}$.
\end{enumerate}
\begin{rem}
As we mentioned above, the weight $w(f)$ of a functional $f$, when
it is defined, is not in general uniquely determined. However, if
$f_1,...,f_{n_{2l-1}}$ is a $(n_{2l-1})$-special sequence, then
for all $i\geq 2$ by $w(f_i)$ we shall put
$m_{\sg(f_1,...,f_{i-1})}$.
\end{rem}
Having defined the set $G$, we define \begin{enumerate} \item[i.]
$\|x\|_G=\sup\{f(x): f\in G\}$, for all $x\in
c_{00}(T\times\nn)$ \item[ii.]
$\xfr_G=\overline{<c_{00}(T\times\nn),\|\cdot\|_G>}$
\end{enumerate}
\begin{rem}\label{remschdec}
The following are easily established.\\
(1) For every $n\in\nn$, the space
$\overline{<(x_{t,n})_{t\in T}>}^{\|\cdot\|_G}$ is isometric to
$X_n$.\\
(2) For every $I,J$ intervals (finite or infinite) of $T$ and
$\nn$ respectively, the projection
\[ P_{I\times J}:\xfr_G\to \xfr_{I\times J}=\overline{
<(x_{t,k})_{t\in I, k\in J}>}^{\|\cdot\|_G} \] has norm one. Consequently we have,
\begin{enumerate}
\item[(a)] The sequence $(X_n)_n$ defines a Schauder
decomposition of $\xfr_G$. \item[(b)] Setting $Z_t=\overline{<
(x_{t,k})_{k\in\nn}>}^{\|\cdot\|_G}$, the sequence $(Z_t)_t$ also
defines a Schauder decomposition of $\xfr_G$.
\end{enumerate}
\noindent (3) Every $j$-block sequence and every $\pi$-block
sequence is a bi-monotone Schauder basic sequence. Hence every
diagonally block sequence is also a bi-monotone basic sequence.
\end{rem}
Next we shall present the basic ingredients for the proof that
certain block sequences in $\xfr_G$ generate HI spaces.
\begin{defn}
Let $x\in c_{00}(T\times\nn)$ and $C>1$. We say that $x$ is a
$C-\ell^1_k$ average if there exists a $j$-block sequence
$x_1\prec_j x_2\prec_j...\prec_j x_k$ such that
$x=\frac{x_1+...+x_k}{k}$, $\|x_i\|_G\leq C$ for $i=1,...,k$ and
$\|x\|_G=1$.
\end{defn}

\begin{defn} \textbf{(RIS)} A $j$-block sequence
$(x_q)_q$ in $\xfr_G$ is said to be a $(C,\epsilon)$ rapidly
increasing sequence, if $\|x_q\|_G \leq C$, and there exists a
strictly increasing sequence $(l_q)_q$ of natural numbers such
that
\begin{enumerate} \item [i.] $\frac{1}{m_{l_{q+1}}}|\spp x_q|<
\epsilon$.\item[ii.] For every $q=1,2,...$ and every $f\in G$ with
$w(f)=m_i,$ $i<l_q$ we have that $|f(x_q)|\leq\frac{C}{m_i}$.
\end{enumerate}
\end{defn}

\noindent \textbf{Notation} \textit{We denote by $D_{Au}$ the
minimal subset of \con satisfying the following properties
\begin{enumerate}
\item[i.] $\pm e_n\in Au$, for all $n\in\nn$. \item[ii.] For every
block sequence $f_1<f_2<...<f_{5n_l}$ in $Au$ we have that
$\frac{1}{m_l}\sum_{i=1}^{5n_l}f_i \in Au$. \item[iii.] $Au$ is
closed under restrictions of its elements on intervals.
\end{enumerate}
We also denote by $Au$ the completion of \con under the norm
induced by the norming set $D_{Au}$}.

We state here a Lemma concerning the behavior of certain averages
of the basis of $Au$. For the proof we refer to Lemma II.9 in
\cite{AT}.

\begin{lem}
\label{auxbasis} Let $l_0\in\nn$ and $h\in D_{Au}$. Then for every
$k_1<...<k_{n_{l_0}}$ we have that \begin{enumerate} \item[i.]
$|h(\frac{1}{n_{l_0}}\sum_{j=1}^{n_{l_0}}e_{k_j})|\leq\frac{2}{m_i\cdot
m_{l_0}}$, if $w(h)=m_i<m_{l_0}$. \item[ii.]
$|h(\frac{1}{n_{l_0}}\sum_{j=1}^{n_{l_0}}e_{k_j})|\leq\frac{1}{m_i}$,
if $w(h)=m_i\geq m_{l_0}$.
\end{enumerate}
If we additionally assume that the functional $h$ admits a tree
analysis $(h_a)_{a\in A}$ such that $w(h_a)\neq m_{l_0}$ for all
$a\in A$, then we have that \begin{enumerate}
\item[i.]$|h(\frac{1}{n_{l_0}}\sum_{j=1}^{n_{l_0}}e_{k_j})|\leq\frac{2}{m_i\cdot
m_{l_0}^2}$, if $w(h)=m_i<m_{l_0}$. \item[ii.]
$|h(\frac{1}{n_{l_0}}\sum_{j=1}^{n_{l_0}}e_{k_j})|\leq\frac{1}{m_i}$,
if $w(h)=m_i\geq m_{l_0}$.
\end{enumerate}
\end{lem}

\begin{prop}
\label{bin} \textbf{(The basic inequality)} Let $(x_q)_q$ be a
$j$-block $(C,\epsilon)$ RIS and let also $(\lambda_q)_q$ be a
sequence of scalars. Then for every $f\in\ G$ we can find $g_1$
such that either $g_1=h_1$ or $g_1=e_k^*+h_1$ with $k\notin \spp
h_1$ where $h_1\in D_{Au}$, $w(f)=w(h_1)$ ,$g_2\in$ $c_{00}(\tr)$
with $\|g_2\|_{\infty}\leq\epsilon$ and $g_1,g_2$ having
nonnegative coordinates such that, \begin{center}
$|f(\sum\lambda_q x_q)|\leq C(g_1+g_2)(\sum |\lambda_q|e_q)$.
\end{center} If we additionally assume that there exists a
$l_0\in\nn$ such that for every $\phi\in G$ with $w(\phi)=m_{l_0}$
and every interval $E$ of the natural numbers,
\begin{center} $|\phi(\sum_{q\in E} \lambda_q x_q)|\leq C(\max_{k\in
E} |\lambda_q|+ \epsilon\sum|\lambda_q|)$, \end{center}
then we can choose $h_1$ to have a tree
analysis $(h_a)_{a\in A}$ such that $w(h_a)\neq m_{l_0}$, for all
$a\in A$.
\end{prop}

We refer the reader to Lemma II.14 of \cite{AT} for a proof of the
above proposition. A direct consequence of the basic inequality
and Lemma \ref{auxbasis} is the following.

\begin{lem}
\label{estimations} Let $(x_q)_{q=1}^{n_{l_0}}$ be a
 $j$-block $(C,\epsilon)$ RIS with $\epsilon\leq \frac{2}{m_{l_0}^2}$.
Then \begin{enumerate} \item[1.] For every $f\in G$ with
$w(f)=m_i$ we have, \begin{enumerate} \item[i.]
$|h(\frac{1}{n_{l_0}}\sum_{j=1}^{n_{l_0}}x_j)|\leq\frac{3C}{m_i\cdot
m_{l_0}}$, if $w(h)=m_i<m_{l_0}$. \item[ii.]
$|h(\frac{1}{n_{l_0}}\sum_{j=1}^{n_{l_0}})|\leq \frac{C}{n_{l_0}}+
\frac{C}{m_i}+ C\epsilon$, if $w(h)=m_i\geq m_{l_0}$.
\end{enumerate}
In particular, $\|\frac{1}{n_{l_0}}\sum_{j=1}^{n_{l_0}}x_j\|\leq
\frac{2C}{m_{l_0}}$. \item[2.] If $(b_q)_{q=1}^{n_{l_0}}$ are
scalars with $|b_q|\leq 1$ for all $q$ such that for every
$\phi\in G$ with $w(\phi)=m_{l_0}$ and every interval $E$ of the
natural numbers we have that,
\begin{center} $|\phi(\sum_{q\in E} \lambda_q x_q)|\leq C(\max_{k\in
E} |\lambda_q|+ \epsilon\sum|\lambda_q|)$.
\end{center}
Then, $\|\frac{1}{n_{l_0}}\sum_{j=1}^{n_{l_0}} b_j x_j\|\leq
\frac{4C}{m_{l_0}^2}$.
\end{enumerate}
\end{lem}

\begin{defn}
Let $x\in c_{00}(T\times\nn)$ and $C>1$. We say that $x$ is a
$C-\ell^1_k$ average if there exists a $j$-block sequence
$x_1\prec_j x_2\prec_j...\prec_j x_k$ such that
$x=\frac{x_1+...+x_k}{k}$, $\|x_i\|_G\leq C$ for $i=1,...,k$ and
$\|x\|_G=1$.
\end{defn}

\begin{lem}
\label{l1toris} Let $x$ be a $C-\ell^1_{n_q}$ average. Then for
every $f\in G$ with $w(f)=m_k<m_q$ we have that
$|f(x)|\leq\frac{1}{m_k} C(1+
\frac{2n_{l-1}}{n_l})\leq\frac{3C}{2}\frac{1}{m_k}$.
\end{lem}
\begin{proof}
Let $x=\frac{1}{n_q}\sum_{i=1}^{n_q} x_i$ be a $C-\ell^1_{n_l}$
average. Let also $f=\frac{1}{m_k}\sum_{i=1}^{n_k}f_i$ with
$(f_i)_{i=1}^{n_k}$ a $j$-block sequence of functionals and
$n_k<n_q$. If we set $E_i=j(\rg f_i)$ and for $l=1,...,n_k$ let
$I_l$ ($J_l$ resp.) be the set of all $i$ such that $j(\spp x_i)$
is contained (resp. intersects) $E_l$. Clearly
$\sum_{l=1}^{n_k}|I_l|\leq n_q$, while for each $l$ we have $\|E_l
x\|\leq\frac{1}{n_q}\sum_{i\in J_l} \|E_lx_i\|\leq\frac{1}{n_q}
C(|I_l|+2)$. Therefore $\sum_{l=1}^{n_k}\|E_lx\|\leq C
\frac{1}{n_q}(\sum_{l=1}^{n_k}|I_l|+2n_k)\leq C(1+
\frac{2n_k}{n_q})$ and the conclusion follows.
\end{proof}

\begin{lem}
\label{l1isris} Let $(x_q)_q$ be a $j$-block sequence in $\xfr_G$
such that each $x_q$ is a $C-\ell^1_{k_q}$ average, where $C>1$
and $k_q$ increasing to infinity and $\epsilon>0$. Then there
exists a subsequence of $(x_q)_q$ which is a
$(\frac{3C}{2},\epsilon)$ RIS
\end{lem}
\begin{proof}
For each $q$ we set $l_q=\max \{l: n_l\leq n_q\}$. There exists a
subsequence of $(x_q)_q$ (we denote this subsequence by $(x_q)_q$
again) such that $(l_q)_q$ is a strictly increasing sequence and
$m_{l_{q+1}}>\frac{1}{\epsilon}|\spp x_q|$ for all $q$. From Lemma
\ref{l1toris} we also get that for each $f\in G$ with
$w(f)=m_k,k<l_q$ we have that $|f(x_q)|\leq
\frac{3C}{2}\frac{1}{m_k}$. Therefore this subsequence is a
$(\frac{3C}{2},\epsilon)$ RIS
\end{proof}

\begin{lem}
\label{hl12} Let $(x_q)_q$ be a $j$-block sequence with each $x_q$
a $C-\ell^1_{k_q}$ average, where $C>1$ and $k_q$ increasing to
infinity. Then for every $l\in\nn$ there exists
$q_1<q_2<...<q_{n_{2l}}$ such that,
\[ \Big\| \frac{x_{q_1}+x_{q_2}+...+x_{q_{n_{2l}}}}{n_{2l}} \Big\|
\leq \frac{3C}{m_{2l}}. \]
\end{lem}

This is a direct consequence of the basic inequality (Proposition
\ref{bin}).

The following holds.

\begin{lem}
\label{hl11} For every $j$-block sequence $(y_n)_n$ and every
$k\in\nn$, there exists a $2-\ell^1_k$ average in $<(y_n)_n>$.
\end{lem}

For the proof we refer to \cite{AT}, Lemma II.22.
Combining Lemma \ref{hl11} and \ref{l1isris} we arrive
at the following.

\begin{lem}
\label{existris} For every $j$-block sequence $(y_n)_n$ in
$\xfr_G$ and for every $\epsilon>0$  there exists a $(3,\epsilon)$
RIS in $<(y_n)_n>$.
\end{lem}

\begin{defn}[\textbf{exact pair}]
A pair $(x,\phi)$ with $x\in c_{00}(T\times\nn)$ and $\phi\in
G$ is said to be a $(C,l)$ exact pair if the following conditions
are satisfied.
\begin{enumerate}
\item[(1)] $1\leq \|x\|_G\leq C$ and for every $f\in G$ with
$w(f)=m_q$ and $q\neq l$ we have that $|f(x)|\leq \frac{3C}{m_q}$
if $q<l$ while $|f(x)|\leq \frac{C}{m^2_l}$ if $q>l$. \item[(2)]
$\phi$ is the result of the $\big(
\aaa_{n_l},\frac{1}{m_l}\big)$-operation and so $w(\phi)=m_l$.
\item[(3)] $\phi(x)=1$ and $\rg x=\rg\phi)$ (we recall that for
$c_{00}(T\times\nn)$, the range of $x$ is the minimal rectangle
generated by intervals that contains $\spp x$).
\end{enumerate}
\end{defn}

The following proposition is a direct consequence of Lemmas
\ref{hl11} and \ref{hl12}.

\begin{prop} If $(x_q)_q$ is a $j$-block sequence, then for every
$l\in\nn$ there exists an $(6,2l)$ exact pair $(x,\phi)$ with
$x\in <(x_q)_q>$ and $\phi\in G$.
\end{prop}
\begin{proof}
From Lemma \ref{existris} we have that there exists
$(y_q)_{q=1}^{n_2l}$ a $(3,\epsilon)$ RIS in $<(x_q)_q>$ with
$\epsilon\leq \frac{1}{m_{2l}^3}$. Choose for each
$q=1,...,n_{2l}$ a $y_q^* \in G$ with $y_q^*(y_q)=1$ and $\rg
 y_q^*\subseteq \rg y_q$. Then the functionals $(y_q^*)_q$ form a
$j$-block sequence and the functional $y^*= \frac{1}{m_{2l}}\sum_q
y_q^*$ is an element of $G$ and if we set
$y=\frac{m_{2l}}{n_{2l}}\sum_{q=1}^{n_{2l}}y_q$ by Proposition
\ref{bin} we get that $(y,y^*)$ is the desired pair.
\end{proof}
\begin{prop}
\label{treelike}\textbf{(The tree like property of special
sequences).} Let $(\phi_i)_{i=1}^{n_{2l-1}}$,
$(\psi_i)_{i=1}^{n_{2l-1}}$ be two distinct special sequences in
$G$. Then \begin{enumerate}\item[i.] For $1\leq i<j \leq n_{2l-1}$
we have that $w(\phi_i)\neq w(\psi_j)$. \item[ii.] There exists $k$
such that $\phi_i=\psi_i$ for all $i<k$ and $w(\phi_i)\neq
w(\psi_i)$ for $i>k$.
\end{enumerate}
\end{prop}

The proof can be readily deduced from the definition of special
sequences. For what follows we restrict ourselves to a specific form of $j$-block sequences. Namely,
\begin{defn}We say that a
$j$-block sequence $(x_n)_n$ is \textit{special $j$-block} if
either $(x_n)_n$ is diagonally block or there exists some
$t\in T$ such that $\spp x_n \subseteq \{t\}\times\nn$ for every
$n\in\nn$.
\end{defn}
\begin{defn}[\textbf{dependent sequences}]
A double sequence $(x_k,\phi_k)_{k=1}^{n_{2l-1}}$ where
$(x_k)_{k=1}^{n_{2l-1}}$ is a special $j$-block sequence and
$\phi_k\in G$ for every $k=1,...,n_{2l-1}$, is said to be a
$(C,2l-1)$ dependent sequence if there exists a sequence
$(2l_k)_{k=1}^{n_{2l-1}}$ of even integers such that the following
conditions are fulfilled.
\begin{enumerate}
\item[(i)] $(\phi_k)_{k=1}^{n_{2l-1}}$ is a $(n_{2l-1})$-special
sequence with $w(\phi_k)=m_{2l_k}$ for all $k=1,...,n_{2l-1}$.
\item[(ii)] Each $(x_k,\phi_k)$ is a $(C,2l_k)$ exact pair.
\end{enumerate}
\end{defn}
\begin{rem}
It is clear that the existence of dependent sequences in certain
subspaces of  $\xfr_G$ is the main tool for proving the HI
property of these subspaces. In the sequel we shall present the
precise statement. Here we want to comment the use of the special
$j$-block sequences in the definition of dependent sequences. A
key ingredient for showing the second inequality in the following
Proposition is the tree-like property satisfied by the
$(n_{2l-1})$-special sequences (Proposition \ref{treelike}).
Nevertheless, when we deal with norms on $c_{00}(\nn)$ then the
tree-like property is also satisfied by all restrictions of the
special sequences on intervals of $\nn$ (see \cite{AT},
Proposition 3.3). However this is not valid when we deal with
$c_{00}(T\times\nn)$ and we consider restrictions on rectangles
generated by intervals of $T$ and $\nn$. Notice that this
problem disappears if we consider special $j$-block sequences and
this is the reason why we introduced this concept.
\end{rem}
\begin{prop}
\label{hp16} Let $(x_k,x^*_k)_{k=1}^{n_{2l-1}}$ be a $(C,2l-1)$
dependent sequence. Then
\begin{equation}
\label{hp16e1} \Big\| \frac{1}{n_{2l-1}} \sum_{k=1}^{n_{2l-1}} x_k
\Big\| \geq \frac{1}{m_{2l-1}}
\end{equation}
\begin{equation}
\label{hp16e2} \Big\| \frac{1}{n_{2l-1}} \sum_{k=1}^{n_{2l-1}}
(-1)^k x_k \Big\| \leq \frac{8C}{m^2_{2l-1}}.
\end{equation}
\end{prop}
\begin{proof}

(1) It can be readily seen that the special functional
$f=\frac{1}{m_{2l-1}} \sum_{k=1}^{n_{2l-1}} x^*_k$ belongs to $G$
thus $f(x_k)\geq \frac{1}{m_{2l-1}}$.\\
(2) First of all it is easy to check that the sequence $(x_k)_{k=1}^{n_{2l-1}}$ is a
$(2C, \frac{1}{n^2_{2l-1}})$ RIS. The inequality follows from
Proposition \ref{bin} after showing that for every $f\in G$ with
$w(f)=m_{2l-1}$ and every interval $E$ we have that,
\begin{eqnarray*}
|f(\sum_{k\in E}(-1)^{k+1}x_k| & \leq & 2C(1+
\frac{2}{m_{2l-1}^2}|E|).
\end{eqnarray*}
To see this choose $f\in G$ with $w(f)=m_{2l-1}$ and observe that such an $f$ must have the following form: $f=\frac{1}{m_{2l-1}}(Fx^*_{t-1}+x^*_t+...+f_{r+1}+...+f_d)$, for
some special sequence
$(x^*_1,x^*_2,...,x^*_r,f_{r+1},...,f_{n_{2l-1}})$ of length
$n_{2l-1}$ with $x^*_{r+1} \neq f_{r+1}$, $w(x^*_{r+1})=w(f_{r+1})$ and $F$ an interval
of the form $[m,\max\spp x^*_{t-1}]$. This representation is a direct consequence of the the tree-like property discussed thoroughly above. We estimate the quantity $f(x_k)$ for each $k$ as follows.

\begin{itemize}
    \item [1.] If $k<t-1$ then $f(x_k)=0$.
    \item [2.] If $k=t-1$ we get $|f(x_{t-1}|=
    \frac{1}{m_{2l-1}}|Fx^*_{t-1}(x_{t-1})|\leq
    \frac{1}{m_{2l-1}}\|x_{t-1}\| \leq \frac{C}{m_{2l-1}}$.
    \item [3.] If $k>r+1$ Proposition \ref{treelike} yields that
    $w(f_i)\neq m_{2l_k}$, for all $i>r$. Using the fact that
    $(x_k,x^*_k)$ is an exact pair and taking into account that
    $n^2_{2l-1}< m_{2l_1}\leq m_{2l_k}$ we proceed in the following manner: \begin{eqnarray*}
    |f(x_k)| & = & \frac{1}{m_{2l-1}}|(f_r+...+f_d)(x_k)|\\
    & \leq & \frac{1}{m_{2l-1}} \big{(}
    \sum_{w(f_i)<m_{2l_k}}|f_i(x_k)|+
    \sum_{w(f_i)>m_{2l_k}}|f_i(x_k)|+ \sum_{2r+2 \leq 2i \leq
    d}|f_{2i}(x_{2k-1})|\big{)}\\
    & \leq & \frac{1}{m_{2l-1}} \big{(}\sum_{2l-1<j<2l_k}
    \frac{3C}{m_j}+ n_{2l-1}\frac{C}{m^2_{2l_k}} \big{)} \leq
    \frac{C}{m^2_{2l-1}}.
    \end{eqnarray*}
    \item [4.] For $k=r+1$ the same argument as in the previous
    case yields that $|f(x_{r+1})| \leq \frac{C}{m_{2l-1}}+
    \frac{1}{m^2_{2l-1}}< \frac{C+1}{m_{2l-1}}$.
\end{itemize}
Let $E$ be an interval. From the above estimates we obtain,
\begin{eqnarray*} |f(\sum_{k\in E}(-1)^{k+1}x_k)| & \leq &
|f(x_{t-1})|+ |\sum_{k\in E\cap [t,r]} \frac{1}{m_{2l-1}}
(-1)^{k+1}|\\
& & +|f(x_{r+1})|+ |\sum_{k\in E\cap [r+2,n_{2l-1}]}f(x_k)|\\
& \leq & \frac{C}{m_{2l-1}}+ \frac{1}{m_{2l-1}}+
\frac{C+1}{m_{2l-1}}+ \frac{C}{m^2_{2l-1}}|E|< 2C(1+
\frac{2}{m^2_{2l-1}}|E|),
\end{eqnarray*}
completing the proof.
\end{proof}
The following is an easy consequence of the previous results.
\begin{prop}
\label{hp18} Let $(x_n)_n$, $(y_n)_n$ be two diagonally block
sequences. Then for every $n\in\nn$ there exists a $(6,2l-1)$
dependent sequence $(z_k,\phi_k)_{k=1}^{n_{2l-1}}$ such that
$z_{2k-1}\in <(x_n)_n>$ and $z_{2k}\in <(y_n)_n>$. Similar results
hold if $(x_n)_n$ and $(y_n)_n$ are $j$-block sequences in the
space $Z_t$ for some $t\in T$.
\end{prop}
We need the following.
\begin{prop}
\label{hp19} Let $Y$ be a subspace of $\xfr_G$. Then one of the
following hold.
\begin{enumerate}
\item[(a)] There exists $n\in\nn$ such that $j_n:Y\to X_n$ is
not strictly singular. \item[(b)] There exists $t\in T$ such
that $\pi_t:Y\to Z_t$ is not strictly singular. \item[(c)] For
every $r>0$, there exists a normalized sequence $(y_n)_n$ in $Y$
and a diagonally block sequence $(w_n)_n$ such that
$\sum_{n\in\nn} \|y_n-w_n\|<r$.
\end{enumerate}
\end{prop}
\begin{proof}
Assume that neither (a) nor (b) hold. Then for every $n\in\nn$,
there exists a subspace $Y'$ of $Y$ such that the map
$j_{\{1,...,n\}}:Y' \to \sum_{i=1}^n \oplus X_n$ is also
strictly singular. The same also holds for the projections
$\pi_{\{t_1,...,t_m\}}$. Hence for every $\ee>0$ and every
$(t,m)\in T \times\nn$ there exists a subspace $Y'$ of $Y$ such
that $\big\| j_{\{1,...,n\}}|_{Y'} \big\|<\ee$ and $\big\|
\pi_{\{t_1,...,t_m\}}|_{Y'} \big\|<\ee$. Using this and a standard
sliding hump argument, we can verify that the third alternative is
satisfied.
\end{proof}
Propositions \ref{hp18} and \ref{hp19} yield the next result.
\begin{cor}
\label{hc20} The following are satisfied.
\begin{enumerate}
\item[(a)] For every $t\in T$ the space $Z_t$ is HI. \item[(b)]
For each diagonally block sequence $(y_n)_n$ the space
$Y=\overline{<(y_n)_n>}$ is HI. \item[(c)] If $Y$ is a subspace of
$\xfr_G$ such that $j_n:Y\to X_n$ and $\pi_t:Y\to Z_t$ are
strictly singular for $(t,k)\in T\times\nn$, then $Y$ is HI.
\end{enumerate}
\end{cor}
\begin{proof}
Parts (a) and (b) are direct consequences of Proposition
\ref{hp18}. To see (c), let $Y$ be a subspace of $\xfr_G$ such
that $j_n:Y\to X_n$ and $\pi_t:Y\to Z_t$ are strictly singular for
every $(t,k)\in T\times\nn$. Let $Y_1$ and $Y_2$ be subspaces of $Y$ and
$\ee>0$. By Proposition \ref{hp19}, there exist normalized block
sequences $(y^1_n)_n$, $(y^2_n)_n$ and a diagonally block sequence
$(w_n)_n$ such that the following are satisfied.
\begin{enumerate}
\item[(1)] For every $n\in\nn$, $y^1_n\in Y_1$ and $y^2_n\in Y_2$.
\item[(2)] $\sum_{n\in\nn} \|w_{2n-1}-y^1_n\|<\ee$ and
$\sum_{n\in\nn} \|w_{2n}-y^2_n\|<\ee$.
\end{enumerate}
The space $W=\overline{<(w_n)_n>}$ is HI by part (b). As $\ee$ can
be chosen arbitrarily small, this shows that
$d(S_{Y_1},S_{Y_2})=0$. As $Y_1$ and $Y_2$ are arbitrary subspaces
of $Y$ we get that $Y$ is HI.
\end{proof}
\begin{prop}
\label{hp22} We have that $\xfr_G^*=\overline{ <\bigcup_{n\in\nn}
X_n^* >}^{\|\cdot\|}$.
\end{prop}
\begin{proof}
Assume not. Then there exist $x^{**}\in \xfr_{G}^{**}$ and $x^*\in
B_{\xfr_G^*}$ such that $\|x^{**}\|=1$, $x^{**}(x^*)>1/2$ and
$\bigcup_n X^*_n\subseteq \mathrm{ker}x^{**}$. Choose a net
$(x_i)_{i\in I}$ in $B_{\xfr_G}$ with
$x_i\stackrel{w^*}{\to}x^{**}$. Clearly we may assume that
\begin{equation}
\label{hp22e3} x^*(x_i)>\frac{1}{2} \text{ for every } i\in I.
\end{equation}
Observe that $j_{\{1,...,n\}}(x_i)\stackrel{w}{\to}0$. Hence
applying Mazur's Theorem and a sliding hump argument, we may
select two sequences $(y_n)_n$ and $(z_n)_n$ such that the
following are satisfied.
\begin{enumerate}
\item[(i)] For every $n\in\nn$, $y_n\in \mathrm{conv}\{x_i:i\in
I\}$. \item[(ii)] $(z_n)_n$ is a $j$-block sequence. \item[(iii)]
$\sum_{n} \|y_n-z_n\|<\frac{1}{8}$.
\end{enumerate}
Notice that for every $n_1<n_2<...<n_k$ we have
\begin{equation}
\label{hp22e7} \Big\| \frac{z_{n_1}+z_{n_2}+...+z_{n_k}}{k}
\Big\|\geq \frac{1}{4}.
\end{equation}
Indeed, by (i) and (\ref{hp22e3}) above we have that
$x^*(y_n)>1/2$ for every $n\in\nn$. Hence by (iii) we get that
$x^*(z_n)>1/4$ for every $n\in\nn$, which clearly implies
(\ref{hp22e7}). Hence we may select a $j$-block sequence $(w_k)_k$
with $w_k=\frac{1}{k}\sum_{n\in F_k} z_k$ where
$F_1<F_2<...<F_k<...$ and each $F_k$ is a finite interval of
$\nn$. As the sequence $(w_k)_k$ is a $j$-block sequence of
$4-\ell^1_k$ averages, Lemma \ref{hl12} yields that for every
$l\in\nn$ there exists $k_1<k_2<...<k_{n_{2l}}$ with
\begin{equation}
\label{hp22e8} \Big\| \frac{1}{n_{2l}} \sum_{i=1}^{n_{2l}} w_{k_i}
\Big\| \leq \frac{12}{m_{2l}}.
\end{equation}
Let $v_l=\frac{1}{n_{2l}} \sum_{i=1}^{n_{2l}} w_{k_i}$. Then $v_l$
is a convex combination of $z_k$'s. Let $v'_l$ be the
corresponding convex combination of $y_n$'s. Then by (i) and
(\ref{hp22e3}) we have $\|v'_l\|>1/2$. By (iii), we get that
$\|v_l-v'_l\|<1/8$. On the other hand, as $m_l\to\infty$ as
$l\to\infty$, by (\ref{hp22e8}) we see that $\|v_l\|\to 0$ and
this leads to a contradiction. The proof is completed.
\end{proof}
\begin{defn}
The HI interpolation space $\hiint$ is the (closed)
subspace of $\xfr_G$ which contains all elements of $\xfr_G$ of
the form $(x,x,...)$.
\end{defn}
\begin{rem}
This definition is an adaptation of the corresponding definition
in \cite{AF}, which in turn follows the scheme of the classical
Davis-Fiegel-Johnson-Pelczynski interpolation method \cite{DFJP}.
\end{rem}

\begin{prop}
\label{hip5} For every $t\in T$ we set
$\bar{e}_t=(e_t,e_t,...)\in\hiint$. Then
$(\bar{e}_t)_t$ becomes a bi-monotone Schauder basis of
$\hiint$.
\end{prop}
\begin{proof}
First we notice that for every $n\in\nn$ we have that $\|a_t
e_t\|_n\leq \frac{1}{2^n}$, for all $t\in T$ Hence $\bar{e}_t\in
\hiint$ for every $t\in T$. Now let $\bar{x}=(x,x,...)\in\hiint$
with $x=\sum_t b_t e_t$. We consider the projection
$\pi_t:\hiint\to Z_t$. We shall show that $\pi_t(\bar{x})=b_t
\bar{e}_t$. Indeed observe that $\{e_{t,k}:k\in\nn\}$ is a
Schauder basis for $Z_t$ (not normalized) and $e_{t,k}^*\big(
\pi_t(\bar{x})\big)= e_{t,k}^*(\bar{x})=b_t$ for every $k\in\nn$.
Hence $\pi_t(\bar{x})=\sum_{k\in\nn} b_t e_{t,k}= b_t\bar{e}_t$.
This easily yields that for every finite interval $I$ of $T$ we
have $\pi_I\big( \hiint \big)= <\{ \bar{e}_t:t\in\I\}>$ and so
$\pi_I(\bar{x})=\sum_{t\in I} b_t \bar{e}_t$. The above argument
and the fact that $\|\pi_I\|=1$ yield that $(\bar{e}_t)_t$ is a
bi-monotone Schauder basis for the space
$Y=\overline{<(\bar{e})_t>}$. It remains to show that $Y$
coincides with $\hiint$. Indeed, let $(I_n)_n$ be the intervals of
$\nn$ such that $I_n=g^{-1}\big(\{1,...n\}\big)$ where $g$ is
defined after Remark \ref{segcint1}. Let also $\bar{x}=(x,x,...)$
with $x=\sum_t b_t e_t$. We claim that the partial sums
$\sum_{t\in I_n} b_t \bar{e}_t$ weakly converge to $\bar{x}$,
which immediately implies the desired result. First we observe
that $\sum_{t\in I_n} b_t \bar{e}_t = \pi_{I_n}(\bar{x})$ and so
$\big\|\sum_{t\in I_n} b_t \bar{e}_t \big\| \leq \|\bar{x}\|$.
Furthermore, for every $x^*\in \bigcup_{n\in\nn} B_{X_n}$ we
have that $x^* \big( \sum_{t\in I_n} b_t \bar{e}_t \big) \to
x^*(\bar{x})$. Proposition \ref{hp22} yields that
$<\bigcup_{n\in\nn} B_{X^*_n}>$ is norm dense in $\xfr^*_G$ and
this proves the claim and the entire proof is completed.
\end{proof}
\begin{notation}\label{maps3}
In the sequel we shall denote by
$J_{X_T}:\hiint\to X_T$ the 1-1, bounded linear
map defined by $J_{X_T}(\bar{x})=x$, where $\bar{x}=(x,x,...)$.
\end{notation}
\begin{prop}
\label{hip7} Let $\hiint$ be the HI interpolation space. Then the following hold:
\begin{enumerate}
\item[(a)] If $Y$ is a closed subspace of $\hiint$
such that $J_{X_T}:Y\to X_T$ is strictly singular, then $Y$ is a
HI space. \item[(b)] If $Y, Z$ are closed subspaces of
$\hiint$ such that $J_{X_T}|_Y$ and $J_{X_T}|_Z$ are
strictly singular, then $d(S_Y,S_Z)=0$.
\end{enumerate}
\end{prop}
\begin{proof}
(a) We observe, following the notation of the previous section,
that for every $t\in T$ the map $\pi_t:\hiint\to
Z_t$ has dimension 1,since as shown in the proof of Proposition
\ref{hip5} $\pi_t\big( \hiint \big)= <\bar{e}_t>$
and so $\pi_t$ is strictly singular. Notice also that for every
$\bar{x}\in \hiint$ and every $n\in\nn$ we have that
$j_n(\bar{x})=J_{X_T}(\bar{x})$. As every $X_n$ is isomorphic
to $X_T$, we get that $j_n|_Y$ is also strictly singular.
Corollary
\ref{hc20}(c) yields the result.\\
(b) We notice that, as in part (a), for every $t\in T$ the maps
$\pi_t|_Y$ and $\pi_t|_Z$ are strictly singular. Moreover, by our
assumptions, for every $n\in\nn$ the maps $j_n|_Y$ and $j_n|_Z$
are also strictly singular. Let $\ee>0$ arbitrary. Arguing as in
Corollary \ref{hc20}(c) we are able to construct two normalized
sequences $(y_n)_n$ and $(z_n)_n$ and a diagonally block sequence
$(w_n)_n$ such that the following are satisfied.
\begin{enumerate}
\item[(i)] For every $n\in \nn$, $y_n\in Y$ and $z_n\in Z$.
\item[(ii] $\sum_n \|w_{2n-1}-y_n\|<\ee$ and $\sum_n
\|w_{2n}-z_n\|<\ee$.
\end{enumerate}
The space $W=\overline{ <(w_n)_n>}$ is HI by Corollary
\ref{hc20}(c). Hence, if we set $W_1=\overline{<(w_{2n-1})_n>}$
and $W_2=\overline{<(w_{2n})_n>}$ we see that
$d(S_{W_1},S_{W_2})=0$. As $\ee$ can be chosen arbitrarily small,
by (ii) above, we conclude that $d(S_Y,S_Z)=0$, as desired.
\end{proof}
\begin{prop}
\label{hip12} The space $\hiint$ is reflexive.
\end{prop}
\begin{proof}
We recall the following well-known facts. First if $T:X\to Y$ is a
Tauberian operator, then $W\subseteq X$ is relatively weakly
compact if and only if $T(W)$ is (see \cite{N2}). Moreover, by a
classical result of A. Grothendieck \cite{Gr}, we have that a set
$K\subseteq X$ is relatively weakly compact if for every $\ee>0$
there exists a weakly compact set $K_\ee\subseteq X$ such that
$K\subseteq K_\ee +\ee B_X$.
 As we have assumed the set $W$ is weakly
compact in $X_T$. It is easy to see that $W$ almost
absorbs $J_{X_T}(B_{\hiint})$, i.e. for every $\ee>0$
there exists $\lambda>0$ such that
$J_{X_T}(B_{\hiint})\subseteq \lambda W + \ee
B_{X_T}$. Hence, by Grothendieck's criterion,
$J_{X_T}(B_{\hiint})$ is a relatively weakly compact
subset of $X_T$. It is a well known fact that $J_{X_T}$ is a
Tauberian operator (c.f \cite{AF}). Hence $B_{\hiint}$ is
also a relatively weakly compact subset of $\hiint$ and
the proof is completed.
\end{proof}

The last step in this section is to prove that if $W$ is a thin subset of $X_T$ then the space $\xfr$ is HI. Let us recall the notion of thin operators.

\begin{defn} \label{thinoper} Let $X,Y$ be Banach spaces and $T:X\to Y$ be a bounded linear operator. $T$ is called a thin operator if $T(B_X)$ is a thin subset of $Y$.
\end{defn}
\begin{rem} It can be readily seen that if $T$ is a thin operator, then it is also strictly singular.(c.f. \cite{AF}).
\end{rem}

Proposition \ref{hip7} immediately yields that if $W$ is a thin subset of $X_T$ then the space $\xfr$ is HI. Combining this with Propositions \ref{hip12}, \ref{hip5} and Remark \ref{remschdec}} we obtain the proof of Theorem \ref{skewthm} stated in the beginning of this section.

\section{ The final results }

\begin{thm}
\label{final1} Every separable reflexive Banach space $X$ is a
quotient of a reflexive  HI Banach space $\hiint(X)$ with a
bimonotone Schauder basis.
\end{thm}
\begin{proof} Let $X$ be a separable reflexive Banach space. By Zippin's Theorem (Theorem \ref{zippin}) we obtain that
 $X$ can be isomorphically embedded into a reflexive Banach space with a bimonotone
Schauder basis $(x_n)_n$. We denote this space $X_Z$. Starting now
with $X_Z$ we pass to the space $\auxa$. The map
$\Phi$ defined in Definition \ref{themapf} yields that $\Phi(\mathrm{K})$ is
$\frac{1}{8}$-dense in the ball of $X_Z$. By passing to the space
$\auxb$ we have by Remark \ref{rint11} that
$J_1(\tilde{\mathrm{K}})=\mathrm{K}$ and thus the operator $\Phi_1=\Phi\circ J_1$
maps $(\tilde{\mathrm{K}})$ onto a  $\frac{1}{8}$-dense in the ball of $X_Z$. We construct the space $\auxx$ starting with $\auxb$ and $\tilde{\mathrm{K}}$. By Theorem \ref{bthinthm} the identity operator $I_{\xi}:\auxx\mapsto
\auxb$ is continuous and maps $\tilde{\mathrm{K}}$
onto itself. Finally, the natural injection (denoted as $J_{X_T}$ in the general case) $J_{\xi}:\hiint\to\auxx$ preserves $\tilde{\mathrm{K}}$ as does $J_1$. Thus,
by taking the composition $Q=\Phi_1 \circ I \circ J_{\xi}$ we can see that
it is an onto map from $\hiint$ to $X_Z$. Thus $X_Z$ is a quotient of $\xfr$. As the set $W_{\xi}$ is thin in $\auxx$ (by Theorem \ref{bthinthm}) Theorem \ref{skewthm} yields that $\xfr$ is the desired reflexive HI space for $X_Z$. A subspace of $X_Z$ will have $X$ as a quotient and this completes the proof.
\end{proof}

Starting with a reflexive $X$ and following the steps of the above proof for the its dual $X^*$, we have the following cofinal property of Indecomposable reflexive Banach spaces within the class of separable reflexive Banach spaces.

\begin{thm}
\label{final2} Every separable reflexive Banach space can be
embedded into an Indecomposable reflexive Banach space.
\end{thm}

\end{document}